\pdfoutput=1
\RequirePackage{ifpdf}
\ifpdf 
\documentclass[pdftex]{sigma}
\else
\documentclass{sigma}
\fi

\usepackage[mathscr]{eucal}
\usepackage[dvipsnames]{xcolor}

\definecolor{red}{rgb}{1,0,0}
\definecolor{purple}{rgb}{0.63, 0.36, 0.94}
\definecolor{Kelly}{rgb}{.1,.65,.1}

\def\Xint#1{\mathchoice
{\XXint\displaystyle\textstyle{#1}}%
{\XXint\textstyle\scriptstyle{#1}}%
{\XXint\scriptstyle\scriptscriptstyle{#1}}%
{\XXint\scriptscriptstyle\scriptscriptstyle{#1}}%
\!\int}
\def\XXint#1#2#3{{\setbox0=\hbox{$#1{#2#3}{\int}$ }
\vcenter{\hbox{$#2#3$ }}\kern-.6\wd0}}

\def\dashint{\Xint-}

\DeclareFontFamily{U}{mathx}{}
\DeclareFontShape{U}{mathx}{m}{n}{<-> mathx10}{}
\DeclareSymbolFont{mathx}{U}{mathx}{m}{n}
\DeclareMathAccent{\widehat}{0}{mathx}{"70}
\DeclareMathAccent{\widecheck}{0}{mathx}{"71}

\def\opE{\mathfrak E} 
\def\opD{\mathfrak D}
\def\opR{\mathfrak R}

\def\rS{{\mathrm S}}
\def\rs{{\mathrm s}}
\newcommand\Leg{ {\mathcal P}} 
\newcommand\Legq{\mathcal Q}
\newcommand\Legu{ {\mathcal P}^1} 

\def\Dx{\delta}		
\def\Fop{\mathtt F} 
\def\Euler{\mathtt E}
\def\Phit{\widetilde\Phi}

\newcommand\uM{\mathsf M} 
\newcommand\uL{\mathsf L} 
\newcommand\uN{\mathsf N} 
\newcommand\uP{\mathsf P} 
\makeatletter 
\newcommand{\uu}[1][\@nil]{%
 \def\tmp{#1}%
 \ifx\tmp\@nnil
 {u}
 \else
 {u}_{#1}
 \fi}
\makeatother

\newcommand\kk{k} 
\newcommand\speed{\beta} 
\def\ri{\mathrm{i}}
\def\vecV{\mathbf V}

\newcommand\VV{\mathbf V}
\newcommand\VW{\mathbf W}
\newcommand\VR{\mathbf R}
\newcommand\VH{\mathbf H}
\newcommand\VZ{\mathbf Z}
\newcommand\qV{\mathcal V}
\newcommand\qW{\mathcal W}
\newcommand\funH{{\mathscr H}} 
\newcommand\funA{{\mathscr A}} 

\newcommand{\dx}{\,{\rm d}x}
\newcommand{\ds}{\,{\rm d}s}
\newcommand{\sdown}{\check{s}} 
\def\gam{\gamma} 
\def\di{\partial}
\def\C{\mathbb C}
\def\R{\mathbb R}
\def\Z{\mathbb Z}
\newcommand\rU{\mathrm U}
\newcommand\bw{\mathbf w}
\newcommand\fsu{\mathfrak{su}}
\newcommand\sj{\mathsf j}
\newcommand\Ad{\operatorname{Ad}}
\newcommand\cl{\mathbf{cl}}
\newcommand\ve{\mathbf e} 
\newcommand\be{{\bf e}}
\newcommand\bp{{\bf p}}
\newcommand\bph{{\bp_H}} 
\newcommand\gamt{\widetilde{\gam}}
\newcommand\gamf{\widehat{\gam}} 
\newcommand\Gamf{\widehat{\Gamma}} 
\newcommand\kh{\widehat{\kk}} 
\newcommand\bert{\big\vert}
\newcommand\rd{\mathrm{d}}
\newcommand\rA{\mathrm{A}}
\newcommand\rC{\mathrm{C}}

\newcommand\tJ{\mathtt{J}}
\newcommand\tU{\mathtt{U}}
\newcommand\tV{\mathtt{V}}
\renewcommand\Re{\operatorname{Re}}
\renewcommand\Im{\operatorname{Im}}
\newcommand\cn{\operatorname{cn}}
\newcommand\dn{\operatorname{dn}}

\newcommand\tomedit[1]{{\color{PineGreen}#1}}

\numberwithin{equation}{section}

\newtheorem{Theorem}{Theorem}[section]
\newtheorem*{Theorem*}{Theorem}
\newtheorem{Corollary}[Theorem]{Corollary}
\newtheorem{Lemma}[Theorem]{Lemma}
\newtheorem{Proposition}[Theorem]{Proposition}
 { \theoremstyle{definition}
\newtheorem{Definition}[Theorem]{Definition}

\newtheorem{Remark}[Theorem]{Remark} }

\begin{document}
\allowdisplaybreaks

\renewcommand{\thefootnote}{}

\newcommand{\arXivNumber}{2308.10125}

\renewcommand{\PaperNumber}{027}

\FirstPageHeading

\ShortArticleName{mKdV-Related Flows for Legendrian Curves in the Pseudohermitian 3-Sphere}

\ArticleName{mKdV-Related Flows for Legendrian Curves \\ in the Pseudohermitian 3-Sphere\footnote{This paper is a~contribution to the Special Issue on Symmetry, Invariants, and their Applications in honor of Peter J.~Olver. The~full collection is available at \href{https://www.emis.de/journals/SIGMA/Olver.html}{https://www.emis.de/journals/SIGMA/Olver.html}}}

\Author{Annalisa CALINI~$^{\rm a}$, Thomas IVEY~$^{\rm a}$ and Emilio MUSSO~$^{\rm b}$}

\AuthorNameForHeading{A.~Calini, T.~Ivey and E.~Musso}

\Address{$^{\rm a)}$~Department of Mathematics, College of Charleston, Charleston, SC 29424, USA}
\EmailD{\href{mailto:calinia@cofc.edu}{calinia@cofc.edu}, \href{mailto:iveyt@cofc.edu}{iveyt@cofc.edu}}

\Address{$^{\rm b)}$~Department of Mathematical Sciences, Politecnico di Torino, Italy}
\EmailD{\href{mailto:emilio.musso@polito.it}{emilio.musso@polito.it}}

\ArticleDates{Received September 26, 2023, in final form March 13, 2024; Published online April 02, 2024}

\Abstract{We investigate geometric evolution equations for Legendrian curves in the 3-sphere which are invariant under the action of the unitary group~${\rm U}(2)$. We define a natural symplectic structure on the space of Legendrian loops and show that the modified Korteweg--de~Vries equation, along with its associated hierarchy, are realized as curvature evolutions induced by a sequence of Hamiltonian flows. For the flow among these that induces the mKdV equation, we investigate the geometry of solutions which evolve by rigid motions in ${\rm U}(2)$. Generalizations of our results to higher-order evolutions and curves in similar geometries are also discussed.}


\Classification{37K10; 37K25; 53E99; 57K33}

\begin{flushright}
\begin{minipage}{70mm}
\it This article is dedicated to Peter J.~Olver\\ in honor of his 70th birthday
\end{minipage}
\end{flushright}

\renewcommand{\thefootnote}{\arabic{footnote}}
\setcounter{footnote}{0}

\section{Introduction}
Broadly speaking, the study of integrable geometric evolution equations for curves has led to fruitful discoveries of the interactions between the structures associated to integrability (e.g., Lax pairs, conservation laws, and B\"acklund transformations) and the geometric and topological features of solution curves. While several authors (including some of us) have extensively studied integrable realizations of the sine-Gordon and nonlinear Schr\"odinger equations in Euclidean geometry \cite{CI99,CI01,LP91}, additional integrable equations arise through investigations of flows in less familiar geometries. For example, the KdV equation appears in many geometric contexts, including flows for curves in the centroaffine plane \cite{CIM1,Pi} and null curves in Minkowski 3-space \cite{MN}, and investigations of curve evolutions in the projective plane and higher-dimensional centroaffine spaces have uncovered geometric flows that realize the Boussinesq and Kaup--Kuperschmidt hierarchies~\cite{CIM2,M1}.

In this article, we study flows for curves in 3-dimensional pseudohermitian Cauchy--Riemann (CR) geometry, specializing to the homogeneous geometry of the 3-sphere. Recall that a CR structure of hypersurface type (sometimes referred to as a pseudoconformal structure) on a~mani\-fold of real dimension $2n+1$ consists of a contact structure together with a compatible almost-complex structure on the $2n$-dimensional contact planes.
(Additional non-degeneracy conditions are usually assumed, but these are vacuous
in the case $n=1$ on which we will focus.)

The 3-sphere inherits its standard CR structure via its embedding as the hypersurface in~$\C^2$ comprised of the set of unit length
of vectors with respect to the Hermitian inner product. The automorphism group of this structure is strictly larger than ${\rm U}(2)$; in fact, the pseudoconformal 3-sphere is preserved by the 8-dimensional group ${\rm SU}(2,1)$ acting on $\C^3$ by preserving a~Hermitian inner product of split signature. In that setting, the 3-sphere is identified with the projectivization of the cone of non-zero null vectors in $\C^3$ and the standard contact structure in~$S^3$ can be expressed in terms of this inner product.

In previous work \cite{CI21,M,MNS, MS}, we investigated
geometric invariants and geometric evolution equations for Legendrian curves as well as curves transverse to the contact distribution in the pseudoconformal 3-sphere.
The setting of this article is {\em pseudohermitian} CR geometry, a sub-geometry of CR geometry in which a contact form is specified, resulting in a compatible
hermitian metric on the contact planes. The canonical connection and curvature for such structures were introduced by Webster \cite{W}. In dimension three, the simply-connected homogeneous pseudo-Hermitian CR manifolds with constant Webster curvature consist of the 3-sphere
$S^3 = \rU(2)/\rU(1)$, the universal cover of anti-de Sitter space $A^3 = \rU(1,1)/\rU(1)$, and the Heisenberg group $H^3$. In each case, the group preserves a fibration to a 2-dimensional space form, where the fibers are transverse to the contact planes.
For the case of the 3-sphere, this is the Clifford map
$\pi_C\colon   S^3 \to S^2$,
a geometrical version of the Hopf fibration (see Section~\ref{geometrysection} for details).

In this article, we will focus on closed Legendrian curves in pseudo-Hermitian $S^3$, in particular their discrete invariants, how their geometry is related
to that of their images
under $\pi_C$, and integrable geometric evolution equations that arise naturally for such curves.
Among other results, we will show that there exists an infinite sequence of geometric evolution equations
\begin{equation}\label{Znflow}
\dfrac{\di \gam}{\di t} = \VZ_n[\gam],\qquad n\ge 1,
\end{equation}
which induce the $n$th flow of the mKdV hierarchy for the curvature of the Legendrian curve $\gam$.

These evolution equations are related to geometric realizations of the mKdV hierarchy that
have already appeared in the literature.
Indeed, applying the Clifford projection to a Legendrian curve in $S^3$ produces
a curve in $S^2$ with the same curvature (up to a factor of 1/2), and curves evolving
by \eqref{Znflow} project to curves evolving by flows previously
identified by Goldstein and Petrich \cite{GP1,GP2} as inducing the mKdV hierarchy (see also the works by Doliwa and Santini \cite{DS} and by Langer and Perline \cite{LP}).
However, the flows we define in this work are geometrically distinctive for two reasons.
First, each flow $\VZ_n$ is Hamiltonian with respect to a symplectic structure on the space of periodic Legendrian curves, which is not induced by any corresponding
structure on $S^2$. Moreover, while curves in $S^2$ can be lifted to Legendrian curves in $S^3$, the Clifford projection does not induce a surjective map from
the space of closed Legendrian curves to closed curves in $S^2$. As we shall see,
a closed curve in $S^2$ must satisfy a rationality condition on its total curvature in order to have a lift into $S^3$ as a closed Legendrian curve, where the resulting rational number is related to the discrete invariants of the lift, as a closed Legendrian curve.\looseness=-1

We now summarize the contents of the paper:
\begin{itemize}\itemsep=0pt
\item[--]
In Section~\ref{sec2}, we review the geometry of Legendrian curves in $S^3$ -- in particular, the moving frame and curvature, and their relation to the projection under $\pi_C$. We define the Legendrian lift of a regular curve
in $S^2$, and compute the lifts of constant curvature circles as an example. We also review the discrete invariants of closed Legendrian curves, and compute such invariants for the circle lifts.
\item[--]
In Section~\ref{sec3}, we define a symplectic structure on the space of periodic Legendrian curves in~$S^3$ (modulo symmetries), and show that Hamiltonian vector field for total length induces mKdV evolution for curvature. More generally, we show that the entire mKdV hierarchy
can be induced by geometric flows for Legendrian curves.
\item[--]
In Section~\ref{sec4}, we compute the stationary curves for the mKdV vector field $Z_1$,
determining which of these are closed and, possibly, periodic in time. For a selection of representative examples, we use the Heisenberg projection, a contactomorphism of the sphere punctured at a point with $\R^3$, to obtain planar projections
that enable us to compute discrete invariants,
such as the Maslov index and Bennequin number (see~\cite{ET1} and the literature therein for an introduction to Legendrian knots and their contact invariants.)
\item[--]
In Section~\ref{sec5}, we discuss some open questions and directions for future research.
\end{itemize}

\section[Legendrian curves in pseudohermitian S\^{}3]{Legendrian curves in pseudohermitian $\boldsymbol{S^3}$}\label{sec2}
\subsection{Moving frames and curvature}\label{geometrysection}

Let $\langle -,-\rangle$ be the standard Hermitian inner product on $\C^2$, and $S^3 \subset \C^2$ denote the set of unit vectors.
 We will think of a regular parametrized curve $\gam\colon \R \to S^3$
as a $\C^2$-valued function of parameter $x$.
Differentiating $\langle \gam, \gam\rangle=1$ shows that $\langle \gam_x, \gam\rangle$ is pure imaginary. The curve $\gam$ is {\it Legendrian} if it satisfies the condition
\begin{equation}\label{Lcond}
\langle \gam_x, \gam \rangle=0.
\end{equation}
We will say that $\gam$ is a unit-speed curve if $\langle \gam_x, \gam_x \rangle =1$ identically, and we will use $s$ in place of~$x$ whenever we are assuming a unit-speed parametrization. We let $\Leg$ denote the set of regular parametrized Legendrian curves in $S^3$, and $\Legu$ denote the subset of those with unit speed.
The group $\rU(2)$ of unitary matrices acts transitively on $S^3$,
preserving the Hermitian inner product,
and thus inducing actions on $\Leg$ and $\Legu$.

Differentiating $\langle \gam, \gam \rangle=1$ and using both the Legendrian condition \eqref{Lcond} and the unit-speed condition
$\langle \gam_s, \gam_s \rangle=1$, shows that the matrix
\[
\Gamma = \begin{pmatrix} \gam \ \gam_s \end{pmatrix}
\]
takes value in $\rU(2)$. We will refer to this as the ${\rm U}(2)$-valued {\it moving frame of $\gam$}.
\begin{Lemma}
The moving frame $\Gamma$ satisfies the Frenet-type equation
\begin{equation}\label{FrenetEquation}
\Gamma_s = \Gamma U, \qquad U = \begin{pmatrix} 0 & -1 \\ 1 & \ri \kk\end{pmatrix},
\end{equation}
where we designate $\kk(s)$ as the {\it curvature function} of $\gam$.
\end{Lemma}
\begin{proof}
Expanding $\gam_{ss}$ in terms of the moving frame gives
\begin{equation}\label{gamss}
\gam_{ss} = -\gam + \ri \kk \gam_s,
\end{equation}
where the first coefficient is determined by differentiating \eqref{Lcond}, and differentiating the unit-speed condition $\langle \gam_s, \gam_s \rangle=1$
implies the second coefficient is purely imaginary. Its imaginary part is computed by $\kk= \Im \langle \gam_{ss}, \gam_s \rangle$.
\end{proof}

For a more general parametrization, we have the following.
\begin{Lemma}
If $\gam$ is a regular Legendrian curve parametrized by an arbitrary variable $x$, its curvature is given by
\begin{equation}\label{generalkform}
 \kk = \dfrac{ \Im \langle \gam_{xx}, \gam_x \rangle}{\quad \langle \gam_x, \gam_x\rangle^{3/2}}.
\end{equation}
Moreover, if $\speed:=\langle \gam_x, \gam_x\rangle^{1/2}$ denotes the {\it speed} of $\gam$,
then the analogue of \eqref{gamss} is
\begin{equation}\label{gamxx}
\gam_{xx} = -\speed^2 \gam + \left( \dfrac{\speed_x}{\speed} + \ri \speed \kk\right) \gam_x.
\end{equation}
\end{Lemma}
\begin{proof} Since any arclength parameter satisfies ${\rm d}s/{\rm d}x=\speed$, then
\[
\gam_s =\speed^{-1} \gam_x , \qquad \gam_{ss}=\speed^{-1}\bigl(\speed^{-1} \gam_x\bigr)_x = \speed^{-2} \gam_{xx}-\speed^{-3} \speed_x \gam_x.
\]
It follows that
\[
\langle \gam_{ss}, \gam_s \rangle=\speed^{-3} \langle \gam_{xx}, \gam_x \rangle -\speed^{-2}\speed_x .
\]
Since $\speed$ is real, we obtain \eqref{generalkform} by taking the imaginary part of each side of the last equation. Using \eqref{gamss}, we write
\[
\speed^{-2} \gam_{xx}-\speed^{-3} \speed_x \gam_x =-\gam +\ri \speed^{-1} k\gam_x,
\]
which, when solved for $\gam_{xx}$, gives \eqref{gamxx}.
\end{proof}

\begin{Definition}
If $\gam$ is periodic with minimal period $L$, we say that $\gam$ is a {\it closed Legendrian curve of period $L$}, and we let $|[\gam]| \subset S^3 $ denote the image or
{\it trace} of $\gam$.
 We let $\Legu_L$ denote the space of closed unit-speed Legendrian curves in $\rS^3$ with fixed length $L$. If the map $\gam\colon S^1_L\to S^3$ is injective, we say that $\gam$ is a (parametrized) {\it Legendrian knot}.
\end{Definition}

We remark that the elements of $\Legu_L$ can be viewed as isometric immersions of the circle $S^1_L$ of circumference $L$ into $S^3$. Also note that the curvature function of a closed curve of period $L$ must satisfy $\kk(x+L)=\kk(x)$, but $L$ is not necessarily the
minimal period of $\kk$.

\subsubsection*{Clifford projections and Legendrian lifts}

The well-known diffeomorphism between $S^3$ and the group ${\rm SU}(2)$ can be defined
by mapping a~unit vector $\bw=(w_1, w_2)^{\mathsf{T}} \in S^3$ to the matrix
\begin{equation}\label{isom}\widehat{\bw}
= \begin{pmatrix} w_1 & -\overline{w_2} \\ w_2 & \hphantom{-}\overline{w_1}\end{pmatrix}\in {\rm SU}(2).\end{equation}
We also identify the Lie algebra $\fsu(2)$ with $\R^3$ via the linear isomorphism
\[\sj\colon \ \begin{pmatrix} \ri x_1 & \ri x_2 + x_3 \\ \ri x_2 - x_3 & -\ri x_1 \end{pmatrix} \mapsto (x_1, x_2,x_3)^{\mathsf{T}}.\]
Combining these allows us to define the 2-to-1 spin-covering homomorphism $\sigma\colon S^3 \to {\rm SO}(3)$ as
\begin{equation}
\label{defofsigma}
\sigma(\bw):= \sj \circ \Ad_{\widehat\bw} \circ~\sj^{-1},
\end{equation}
where $\Ad$ denotes the adjoint representation of ${\rm SU}(2)$, and we take the standard inner product on $\R^3$.

We define the {\it Clifford map} $\pi_C\colon S^3 \to S^2$ in terms of $\sigma$ as follows. Let $\{ \ve_1, \ve_2, \ve_3\}$ be the standard basis on $\R^3$;
then
\[\pi_C(\bw) := \sigma(\bw) \ve_1.\]
Since $\sigma\bigl({\rm e}^{\ri \theta}\bw\bigr)$ differs from $\sigma(\bw)$ by rotation that fixes $\ve_1$,
this endows $S^3$ with the structure of a circle bundle whose fibers are the integral curves of the {\it characteristic vector field} $\bw\mapsto \ri\bw$. Thus, condition \eqref{Lcond} can be interpreted as saying that a curve in $\rS^3$ is Legendrian if and only if it is orthogonal to the fibers of the Clifford map.

\begin{Definition}
The {\it Clifford projection} of a Legendrian curve $\gam$ is the immersed curve ${\eta\colon\R\!\to\! S^2}$ defined by $\eta =\pi_C \circ \gam$. (Note that $\gam$ is regular if and only if $\eta$ is regular.)
\end{Definition}

\begin{Proposition}
Let $\gam$ be a Legendrian curve parametrized by arclength $s$, with curvature function $\kk(s)$. Then its Clifford projection has speed~$2$ and Frenet curvature $\kk/2$.
\end{Proposition}

\begin{proof} Let $\bw=(w_1, w_2)^{\mathsf{T}}$ be the vector of components of $\gam$. By applying \eqref{defofsigma} to the standard basis, we can write $\sigma(\bw)$ in matrix form as \begin{equation}\label{sigmawform}\renewcommand{\arraystretch}{1.2}
\sigma(\bw) = \begin{pmatrix} |w_1|^2 - |w_2|^2 & -w_1 w_2 & -\overline{w_1} \overline{w_2} \\
w_1\overline{w_2} + \overline{w_1} w_2 & \tfrac12\bigl(w_1^2-w_2^2\bigr) & \tfrac12\bigl(\overline{w_1}^2-\overline{w_2}^2\bigr) \\
\ri (w_1\overline{w_2} - \overline{w_1} w_2) & \tfrac12\ri \bigl(w_1^2+w_2^2\bigr) & -\tfrac12 \ri \bigl(\overline{w_1}^2 + \overline{w_2}^2\bigr)
\end{pmatrix} \begin{pmatrix} 1 & 0 & \hphantom{-}0 \\ 0 & 1 & -\ri \\ 0 & 1 & \hphantom{-}\ri \end{pmatrix},
\end{equation}
and compute
\[
\sigma(\bw)_s = \sigma(\bw) \begin{pmatrix} 0 & -2\cos \alpha & -2\sin\alpha \\ 2\cos\alpha & 0 & 0\\ 2\sin\alpha & 0 & 0\end{pmatrix},
\]
where $\cos \alpha + \ri \sin \alpha = w_1 w_2' - w_1' w_2$. (The fact that this has unit modulus follows from the Legendrian condition $w_1\overline{w_1}'+w_2 \overline{w_2}'=0$ and the unit-speed condition $w_1'\overline{w_1}'+w_2' \overline{w_2}'=1$.)
The last equation shows that $\sigma(\mathbf w)$ gives a parallel orthonormal along the Clifford projection.

In order to construct a Frenet frame for the Clifford projection, we modify $\sigma(\bw)$. From \eqref{sigmawform}, one can check that
\begin{equation*}
\sigma\bigl({\rm e}^{\ri \phi} \bw\bigr) = \sigma(\bw) \begin{pmatrix} 1 & 0 & 0 \\ 0 & \cos(2\phi) & \sin(2\phi) \\ 0 & -\sin(2\phi) & \cos(2\phi) \end{pmatrix}.
\end{equation*}
(Note that the first column of $\sigma(\bw)$, which gives the value of the Clifford map, is unchanged.)
In particular, setting $\phi = -\alpha/2$ and differentiating this formula shows that $F= \sigma\bigl({\rm e}^{-\ri \alpha/2} \bw\bigr)$ satisfies
\[F_s = F \begin{pmatrix} 0 & -2 & 0 \\ 2 & \hphantom{-}0 & -\alpha_s \\ 0 & \hphantom{-}\alpha_s & 0 \end{pmatrix}.\]

On the other hand, the component-wise form of \eqref{gamss} gives
\[
w_1''=-w_1+\ri\kk w_1', \qquad w_2''=-w_2+\ri\kk w_2',
\]
from which we compute
\[\ri \alpha_s = \frac{(w_1 w_2' - w_1' w_2)'}{ w_1 w_2' - w_1' w_2}=\frac{w_1 w_2'' - w_1'' w_2}{ w_1 w_2' - w_1' w_2} = \ri \kk.\]
This gives $\alpha_s=\kk$ and the result follows.
\end{proof}

\begin{Remark}\label{LiftingRemark}
Conversely, suppose $\eta\colon\R \to S^2$ is a curve parametrized by $x$ with constant speed~2 and curvature function $\kk/2$. Identify ${\rm SO}(3)$ with the oriented orthonormal frame bundle of $S^2$, with the basepoint map ${\rm SO}(3)\to S^2$ is given by the first column. Then the Frenet frame $F=(\eta, T, N)$ is a lift of $\eta$ into ${\rm SO}(3)$ satisfying
\[F_x = F \begin{pmatrix} 0 & -2 & \hphantom{-}0 \\ 2 & \hphantom{-}0 & -\kk \\ 0 & \hphantom{-}\kk & \hphantom{-}0 \end{pmatrix}.\]
In turn, let $\gamt\colon\R \to S^3$ be a lift of $F$ relative to the double cover $\sigma$ \big(i.e., such that $F = \sigma\circ \gamt$\big), which is unique
up to a minus sign. Then, if $\alpha(x)$ is an antiderivative of $\kk(x)$,
$\gam(x) = {\rm e}^{\ri \alpha/2} \gamt(x)$ is a unit-speed Legendrian lift of $\eta$, unique up to multiplication by a unit modulus constant.
\end{Remark}

\subsection[Closed Legendrian curves in S\^{}3 and their discrete invariants]{Closed Legendrian curves in $\boldsymbol{S^3}$ and their discrete invariants}

The classical invariants of closed Legendrian curves and Legendrian knots are the Maslov index (or rotation number) and the Bennequin invariant (see, e.g., \cite{ET1}).
We first discuss how these are computed, before passing to less familiar discrete invariants.

For a given unit vector $\bw = (w_1, w_2)^{\mathsf{T}} \in S^3$, let $\bw^* = (-\overline{w_2}, \overline{w_1})^{\mathsf{T}}$ (i.e., the second column in \eqref{isom}).
 Then vector fields $\bw \mapsto \bw^*$ and $\bw \mapsto \ri \bw^*$ give an orthogonal parallelization of the contact distribution on $S^3$.
Hence, for any $\gam \in \Legu_L$ there is a unique smooth function $\theta\colon\R/(L\Z)\to \R/(2\pi\Z)$ such that
\begin{equation}\label{Masloveq}
\gam_s={\rm e}^{\ri\theta}\gam^*,
\end{equation}
and the Maslov index ${\boldsymbol{\mu}}_{\gamma}$ is the degree of $\theta$.

If $|[\gam]|$ is a Legendrian knot ${\mathcal K}$, we define the Bennequin invariant ${\bf tb}_{\gamma}$ by the following construction.
For $\varepsilon\in \R$, let ${\mathcal K}_{\varepsilon}={\rm e}^{2\pi \ri\varepsilon}{\mathcal K}$. Then there exists
a connected open interval $I \subset (0,1)$ such that ${\mathcal K}\cap {\mathcal K}_{\varepsilon}=\varnothing$ for every $\varepsilon\in I$, and the Bennequin invariant ${\bf tb}_{\gamma}$ equals the linking number ${\rm Lk}({\mathcal K}, {\mathcal K}_{\varepsilon})$.

\begin{Definition}
Let $\eta$ be the Clifford projection of $\gam\in \Legu_L$ and let $L_{\eta}$ be its length. Since the Clifford map doubles the speed of the curve, there is a positive integer $\cl_{\gamma}$ called the {\it Clifford index}
such that $2L=\cl_{\gamma}L_{\eta}$.

We say that $\gam$ has {\it spin $1$} if $\sigma\circ \gam$ has trivial homotopy class in $\pi_1({\rm SO}(3)) =\Z/2$, or has {\it spin~$1/2$} if the homotopy class is non-trivial. Similarly, we say a closed curve in $S^2$ has spin~1 if its Frenet frame $F$ has trivial homotopy class,
and has spin~1/2 if the Frenet frame has non-trivial homotopy class.
\end{Definition}

\begin{Proposition}\label{Maslov} Let $\gam\in \Legu_L$. Then
\[{\boldsymbol \mu}_{\gamma}=\frac{1}{2\pi}\int_0^L \kk(s)\,{\rm d}s.\]
Hence, the total curvature\footnote{For the sake of convenience, for us the `total curvature' will always mean the integral of the curvature function with respect to arclength, {\em divided by} $2\pi$.} of a closed Legendrian curve is an integer. In addition, this implies that the total curvature of the Clifford projection is the rational number ${\boldsymbol \mu}_{\gamma}/\cl_{\gamma}$.
\end{Proposition}

\begin{proof} Differentiating \eqref{Masloveq}, we get
\[\gam_{ss}=\ri\theta_s \gam_s+{\rm e}^{\ri\theta}(\gam^*)_s=\ri\theta_s \gam_s+{\rm e}^{\ri\theta}(\gam_s)^*=\ri\theta_s \gam_s-\gam.\]
Comparing with \eqref{gamss} shows that $\theta_s=\kk$. Hence, $\operatorname{deg} (\theta)=\frac{1}{2\pi}\int_0^L \kk(s) \, {\rm d}s$.
On the other hand, if we use $\sdown$ to denote the arclength parameter along the Clifford projection, then $\sdown=2s$ and $\kappa(\sdown) = \tfrac12 \kk(2s)$, so
\[\int_0^L \kk(s)\,{\rm d}s = \int_0^{\cl_\gamma L_\eta} \kappa(\sdown)\,{\rm d}\sdown = \cl_\gamma \int_0^{L_\eta} \kappa(\sdown)\,{\rm d}\sdown,\]
and the second assertion follows.
\end{proof}

\begin{Remark}
Conversely, let $\eta$ be a closed curve in $S^2$ of length $2T$, with total curvature equal to $a/b$, where $a\in \Z$ and $b\in \Z^+$ have no common divisors. Let $\gam$ be a unit-speed Legendrian lift of $\eta$ constructed as described in Remark \ref{LiftingRemark}. Then
\[\gam(s + b T) = \left\{ \begin{aligned} &(-1)^a \gam(s) & &\text{if $\eta$ has spin $1$,} \\
							&(-1)^{a+b} \gam(s) & &\text{if $\eta$ has spin $1/2$.}
\end{aligned}
\right.
\]
Hence, in the case of spin 1, the Clifford index is $\cl_\gamma = b$ if $a$ is even and $\cl_\gamma=2b$ if $a$ is odd.
In the case of spin $1/2$, $\cl_\gamma=b$ if $a+b$ is even and $\cl_\gamma=2b$ if $a+b$ is odd.
\end{Remark}

\begin{Definition} The CR-analogue of stereographic projection is the {\it Heisenberg projection} $\bp_{H}$, a {\em contactomorphism} from the punctured
sphere to $\R^3$ with the contact form ${\rm d}z-y{\rm d}x+x{\rm d}y$.
Taking ${\rm S}=(-1,0)\in \C^2$ as the omitted point, then for $\bigl(z^1,z^2\bigr)\in S^3\setminus \{{\rm S}\}$,
\begin{equation}\label{heisenbergprojection}
\bph\colon \ \bigl(z^1,z^2\bigr)\mapsto \biggl(\mathrm{Re \biggl(\frac{\ri\sqrt{2}z_2}{1+z_1}\biggr),
\mathrm{Im}\biggl(\frac{\ri\sqrt{2}z_2}{1+z_1}\biggr)},\mathrm{Re}\biggl(\frac{\ri(1-z_1)}{1+z_1}\biggr)\biggr)^{\mathsf{T}}.
\end{equation}
\end{Definition}

\begin{Remark}
The Maslov index and the Bennequin invariant can be computed directly from the Heisenberg projection $\gam_H={\bf p}_H\circ \gam$, assuming that ${\rm S}\notin |[\gam]|$.
The Maslov index is the turning number of the {\it Lagrangian projection} of $\gam_H$, defined by $\alpha = \pi_z\circ \gam_H$, where $\pi_z$ denotes the orthogonal projection from $\R^3$ onto the $xy$-coordinate plane. If $\gam$ is a knot, the Bennequin number is the writhe of $\alpha$ with respect to upward oriented $z$-axis (see, e.g., \cite{ET1}).
\end{Remark}

\subsection{Constant-curvature examples}
To illustrate the Legendrian lift, we consider lifting circles obtained
by intersecting the unit sphere $S^2 \subset \R^3$ with the plane $x=h$.
Let $\rC(h) = \bigl\{ (x,y,z) \in S^2 \mid x=h\bigr\}$ for $h \in (-1,1)$. Since~$\rC(h)$ has
radius $\ell = \sqrt{1-h^2}$ and constant curvature $h/\ell$ as a curve on $S^2$,
the total curvature of $\rC(h)$ is equal to $h$. (Note that the curvature may be negative due to choice of orientation.) Thus, the only circles with closed Legendrian lifts are those for which $h$ is rational.

\begin{Proposition}\label{constantcurvature} For a pair of relatively prime positive integers $m$, $n$, let $\rC_{m,n}$ denote be the circle $\rC(h)$ for $h = (m-n)/(m+n)$.
Then $\rC_{m,n}$ has spin $1/2$, and its Legendrian lifts are left-handed torus knots of type $(-m,n)$. Moreover, these lifts have constant curvature $\kk_{m,n}=(m-n)/\sqrt{mn}$, total curvature equal to $m-n$, and Clifford index $m+n$.
\end{Proposition}

\begin{figure}[t]\centering
\includegraphics[height=5cm,width=5cm]{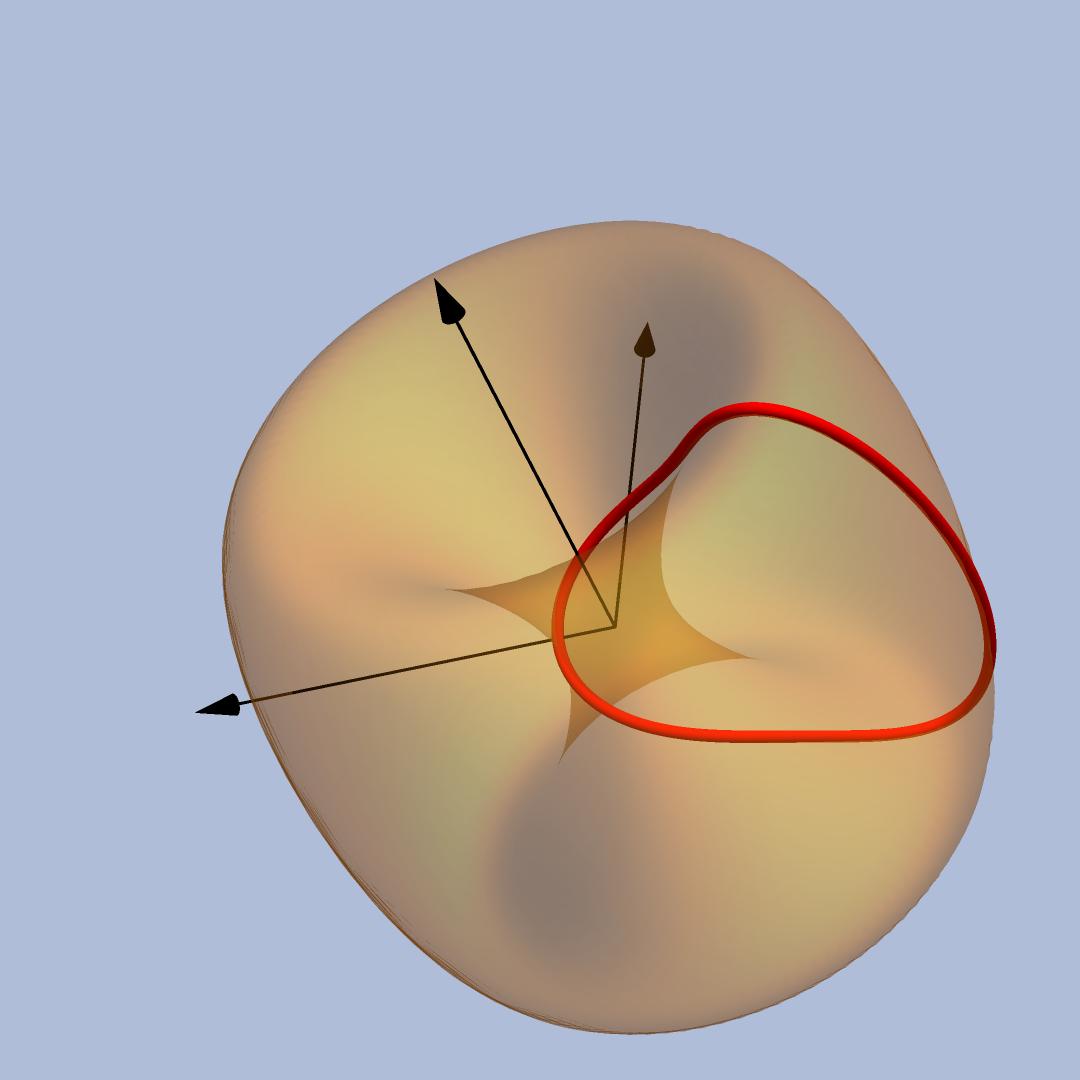}\qquad
\includegraphics[height=5cm,width=5cm]{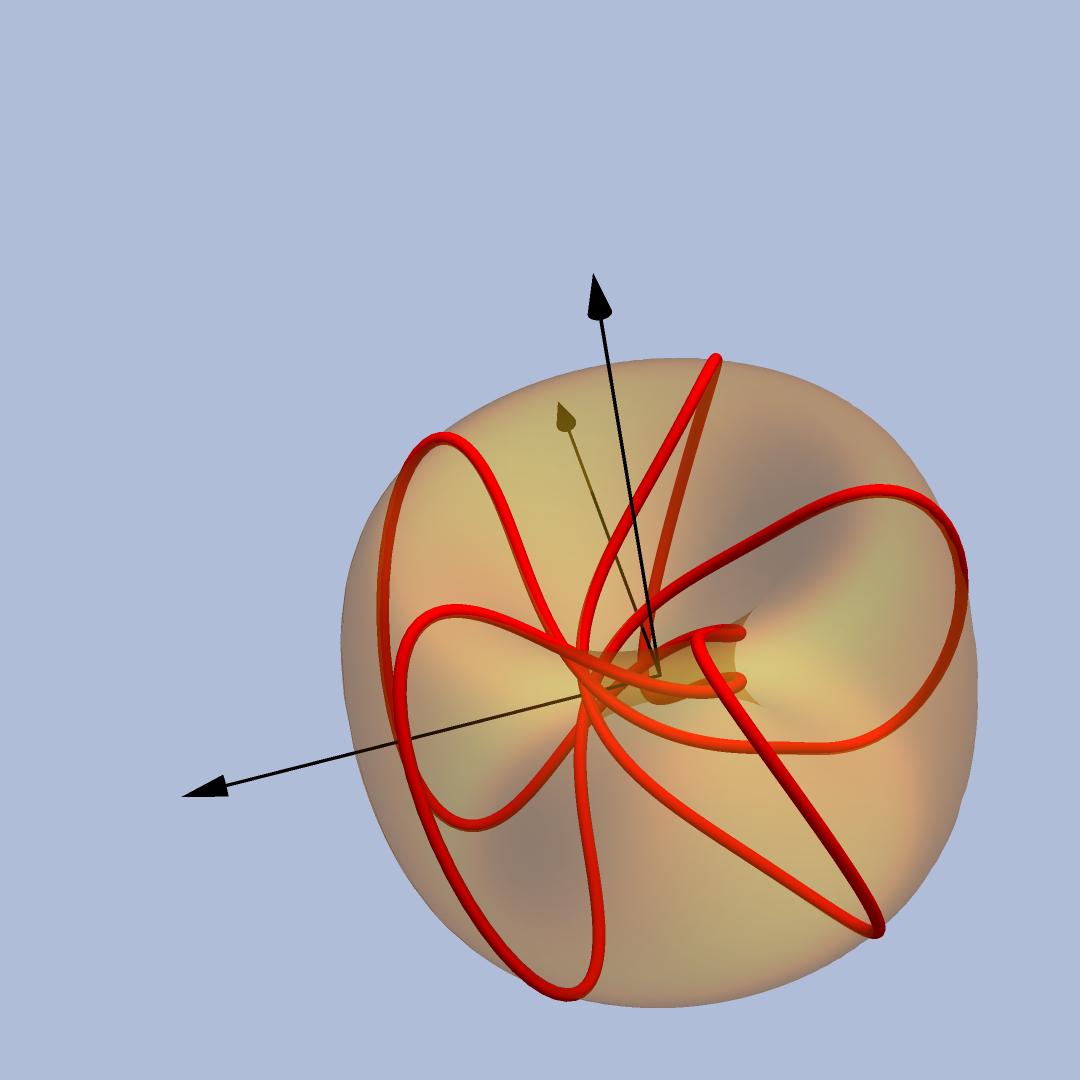}
\caption{Left: the Heisenberg projection of the Legendrian knot $\gam_{1,1}$, a topologically trivial knot with Maslov index $0$ and Bennequin invariant $-1$.
Right: the Heisenberg projection of the Legendrian knot~$\gam_{3,5}$, a torus knot of type $(-3,5)$ with Maslov index $2$ and Bennequin invariant $-15$. The tori are the Heisenberg projections of ${\mathcal T}_{m,n}\subset S^3$, $m=n=1$ (left) and $m=3$, $n=5$ (right).}
\end{figure}

\begin{proof} Let $\gam_{m,n}\colon\R\to S^3$ be defined by
\begin{equation}\label{gammn}\gam_{m,n}(s)=\frac{1}{\sqrt{m+n}}\bigl(\sqrt{m}{\rm e}^{-\ri ns/{\sqrt{mn}}}, \sqrt{n}{\rm e}^{\ri ms/{\sqrt{mn}}}\bigr).
\end{equation}
One can check that $\gam_{m,n}$ is a unit-speed Legendrian curve with curvature $\kk_{m,n}=(m-n)/\sqrt{mn}$ and least period $2\pi\sqrt{mn}$. Thus, the total curvature of $\gam_{m,n}$ is $m-n$. By construction, $\gam_{m,n}$~is a~closed solenoidal curve of type $(-m,n)$ contained in the embedded torus ${\mathcal T}_{m,n}\subset S^3$ parametrized by
\[f_{m,n}(\theta_1,\theta_2)=\frac{1}{\sqrt{m+n}}\bigl( \sqrt{m}{\rm e}^{\ri \theta_1}, \sqrt{n}{\rm e}^{\ri \theta_2}\bigr).
\]
The Clifford projection of $\gam_{m,n}$ is
\[\eta_{m,n}(s)=(h, \ell \cos(2s/\ell),\ell \sin(2s/\ell)),
\]
where we define $\ell:=2\sqrt{mn}/(m+n)$ for short.
Thus, $\eta_{m,n}$ is a parametrization with constant speed $2$ of $\rC_{m,n}$. The least period of $\eta_{m,n}$ is $\pi\ell$, while the least period of $\gam_{m,n}$ is $2\pi\sqrt{mn}$; it follows from Proposition~\ref{Maslov} that $m+n$ is the Clifford index of $\gam_{m,n}$.

As in the proof of Proposition~\ref{Maslov},
we use $\sdown = 2s$ to denote the arclength parameter along the Clifford projection, and let $\check{\eta}_{m,n}(\sdown)=\eta_{m,n}(\sdown/2)$ be its unit-speed reparametrization.
This has least period $2\pi \ell$ and curvature $h/\ell$. Thus, its Frenet frame is of the form
$F_{m,n}(\sdown)=\rA \exp \left(\sdown {\rm M}_{m,n}\right)$,
where $\rA$ is some fixed matrix in ${\rm SO}(3)$ and
\[{\rm M}_{m,n}=\begin{pmatrix} 0&-1&0\\1&\hphantom{-} 0&-h/\ell \\0&\hphantom{-}h/\ell &0\end{pmatrix}.\]
Since the Lie algebra isomorphism induced by the spin-covering homomorphism $\sigma$ is
\[ \begin{pmatrix} \ri b^1_1&-a^2_1+\ri b^2_1\\ a^2_1+\ri b^2_1&-\ri b^1_1 \end{pmatrix} \mapsto
\begin{pmatrix} 0 & -2a^2_1 & -2b^2_1\\ 2a^2_1 & \hphantom{-}0 & \hphantom{-} 2b^1_1\\ 2b^2_1 &-2 b^1_1&\hphantom{-}0 \end{pmatrix},
\]
the lift of $F_{m,n}$ to ${\rm SU}(2)$ is $\hat{F}_{m,n}(\sdown) = \hat{{\rA}} \exp(s\hat{{\rm M}}_{m,n})$, where $\sigma(\hat{\rA}) = \rA$ and
\[\hat{{\rm M}}_{m,n}=\begin{pmatrix} -\ri h/(2\ell) & -1/2 \\ 1/2 &\ri h/(2\ell) \end{pmatrix}.
\]
On the other hand, the eigenvalues of $\hat{\rm M}_{m,n}$ are $\pm \ri/(2\ell)$, so it follows that the least period of~$\hat{F}_{m,n}$ is $4\pi\ell$, which is twice the least period of $F_{m,n}$. This proves that $\rC_{m,n}$ has spin $1/2$.
\end{proof}

\begin{Remark}
Proposition~\ref{constantcurvature} is an adaptation to the context of pseudo-Hermitian geometry of the classification of Legendrian curves with constant CR-curvature given in \cite{MS}. It follows from Propositions ~\ref{Maslov} and~\ref{constantcurvature} that the Maslov index of $\gam_{m,n}$ is $m-n$.
Computing the writhe of the Lagrangian projection of $\bph \circ \gam_{m,n}$,
we find that the Bennequin invariant of $\gam_{m,n}$ is~$-mn$. Thus, according to the classification of the contact isotopy classes of Legendrian torus knots~\cite{EH}, closed Legendrian curves with constant curvature provide explicit models for negative torus knots with maximal Maslov index and maximal Bennequin number.
\end{Remark}

\begin{Remark}
We also note a rather curious fact: the inversion of the Lagrangian projection of $\bph\circ \gamma_{m,n}$ with respect to the origin is an epicycloid traced by the path of a point at distance $\bigl(\sqrt{m}+\sqrt{m+n}\bigr)/\sqrt{2n}$ from the origin, generated by rolling a circle of radius $r=m/\sqrt{2n(m+n)}$ along a fixed circle of radius \smash{$R=n/\sqrt{2n(m+n)}$} centered at the origin (see Figure \ref{FIG1S1}). Indeed, from \eqref{heisenbergprojection} and \eqref{gammn}, it follows that the Lagrangian projection of $\bph \circ \gam_{m,n}$ is
\[
\alpha_{m,n}(u)=\frac{1}{\varrho_{m,n}(u)}\biggl(-c_{m,n}\sin u-d_{m,n}\sin\biggl(\frac{m+n}{m}u\biggr),c_{m,n}\cos u+d_{m,n}\cos\biggl(\frac{m+n}{m}u\biggr)\biggr),
\]
where $c_{m,n}=\sqrt{2n(m+n)}$, $d_{m,n}=\sqrt{2m n}$, $\varrho_{m,n}(u)=2m+n+2\sqrt{m(m+n)}\cos(n u/m)$ and we have made the change of
variable $u=\sqrt(m/n)s$.
Thus, the inversion $\beta_{m,n}=\|\alpha_{m,n}\|^{-2}\alpha_{m,n}$ is
\[
\beta_{m,n}(u)=\biggl(-\hat{c}_{m,n}\sin u-\hat{d}_{m,n}\sin\biggl(\frac{m+n}{m}u\biggr),\hat{c}_{m,n}\cos u+\hat{d}_{m,n}\cos\biggl(\frac{m+n}{m}s\biggr)\biggr),
\]
where $\hat{c}_{m,n}=\sqrt{m+n}/\sqrt{2n}=R+r$ and $\hat{d}_{m,n}=\sqrt{m}/\sqrt{2n}$.
By comparison, the epicycloid traced by the path of the point $(a+R+r,0)$, generated by rolling of a circle with radius $r$ on the fixed circle of radius $R$ centered at the origin, can be parametrized by
\[
\eta_{R,r,a}(u)=\biggl((R+r)\cos u+a\cos\biggl(\frac{R+r}{r}u\biggr),(R+r)\sin u+a\sin\biggl(\frac{R+r}{r}u\biggr)\biggr).
\]
Hence, if we set $a=\hat{d}_{m,n}$, then $|[\beta_{m,n}]$ is obtained by rotating the epicycloid by $\pi/2$ about the origin.
\end{Remark}

\begin{figure}[t]\centering
\includegraphics[height=5cm,width=5cm]{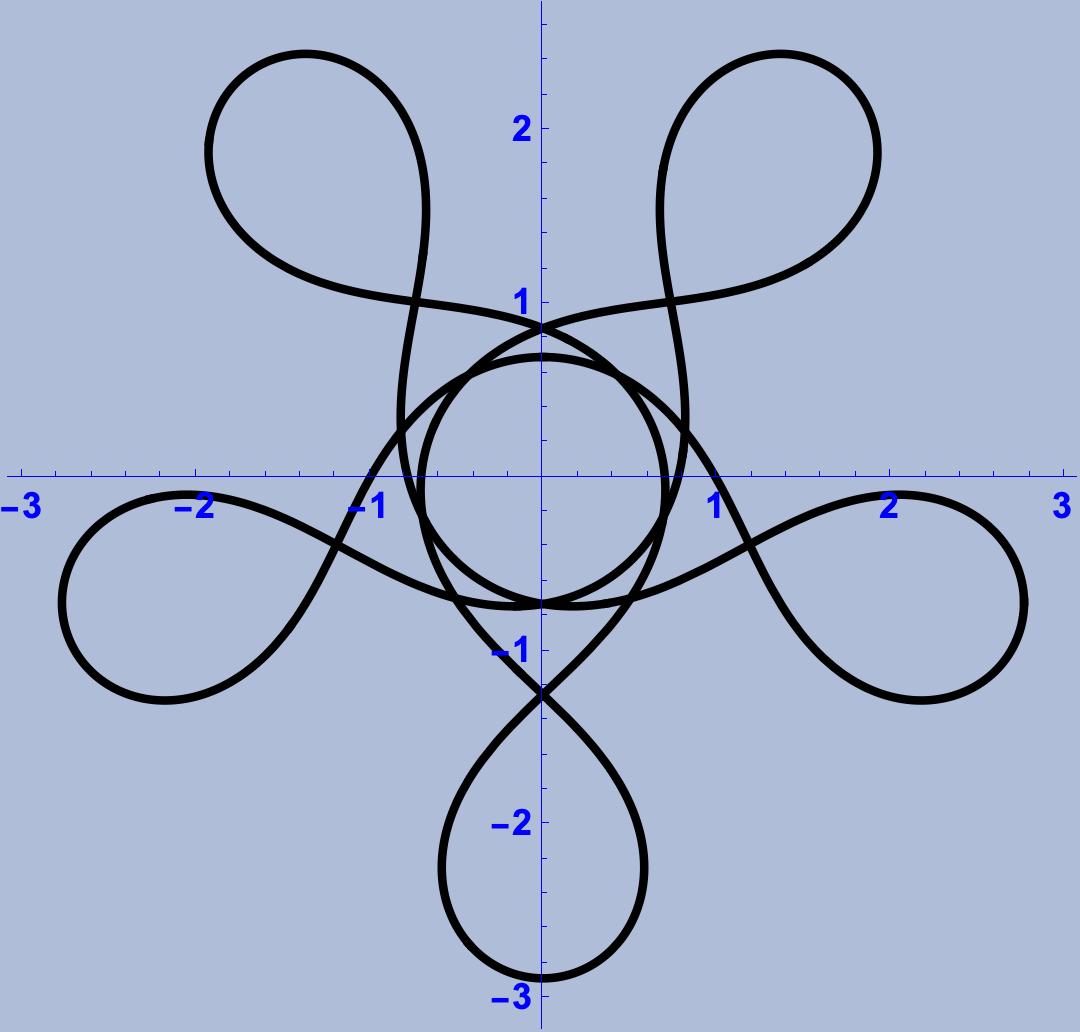}\qquad
\includegraphics[height=5cm,width=5cm]{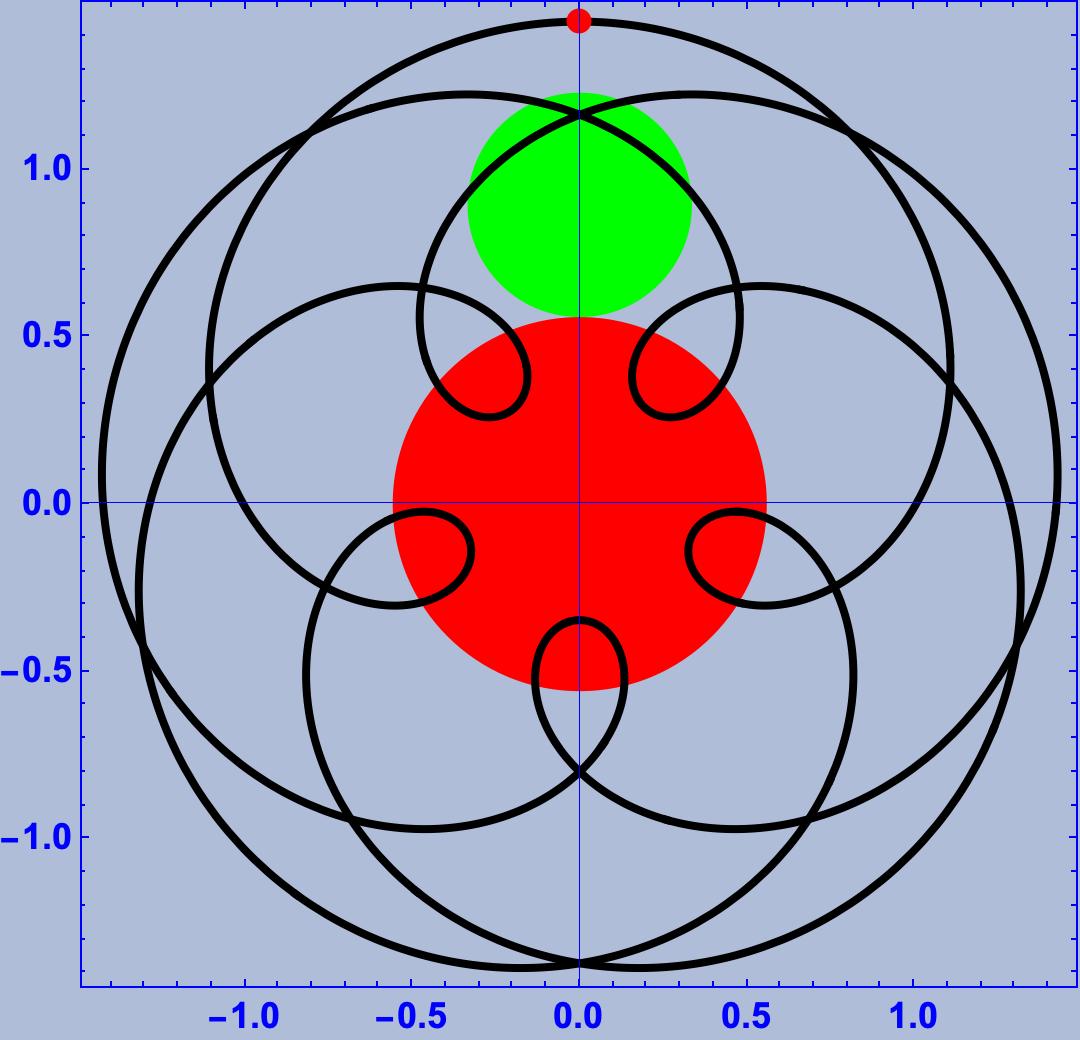}
\caption{Left: the Lagrangian projection $\alpha_{3,5}$ of $\bph \circ \gam_{3,5}$.
Its turning number is $2$, and there are fifteen ordinary double points, each with intersection index $-1$. Right: the epicycloid obtained inverting~$\alpha_{3,5}$ with respect to the origin.}\label{FIG1S1}
\end{figure}

\section{mKdV-type flows for Legendrian curves}\label{sec3}
\subsection{Symplectic structure}

Recall that $\Leg_L$ denotes the space of regular {\em periodic} parametrized Legendrian curves ${\gam\colon \R \to S^3}$ with period $L$ in parameter $x$. The space $\Leg_L$ has the structure of an infinite-dimensional manifold.\footnote{It can be shown that $\Leg_L$ has a Fr\'echet manifold structure by adapting the argument used in Section~43.19 of Kriegl and Michor's monograph~\cite{KM97} for the group of contact diffeomorphisms.
See also Lerario and Mondino~\cite{LM19}, where a Hilbert manifold structure is discussed, and Haller and Vizman~\cite{HV22}, where it is proven that the space of weighted Legendrian knots in a contact 3-manifold is a co-adjoint orbit of the group of contact diffeomorphisms, endowed with a Fr\'echet manifold structure. } We begin by characterizing its tangent spaces.

\begin{Lemma}\label{tanque}
Assume $\gam \in \Leg_L$ and let $\gamf(x,t)$ be a variation of $\gam$, i.e., $\gamf(x,t)$ is smooth, belongs to~$\Leg_L$ for every fixed $t$ $($for $|t|$ sufficiently small$)$ and $\gamf(x,0)=\gam(x)$. Let
$\VV = \frac{\partial \gamf(x,t)}{\partial t}\Big|_{t=0}.$
Then the variation vector field $\VV$ along $\gam$ is of the form
\[\VV = p\gam_x + q (\ri\gam_x) + r (\ri \gam),\]
where $p$, $q$, $r$ are real-valued $L$-periodic functions of $x$ satisfying $r_x = 2q \langle\gam_x, \gam_x\rangle.$
\end{Lemma}
\begin{proof}
We compute
\begin{align*}
0=\di/\di{t}\bert_{t=0} \bigl\langle \gamf_x, \gamf \bigr\rangle={}& \langle \VV_x, \gam\rangle + \langle \gam_x, \VV \rangle \\
={}& \langle (p_x+\ri q_x) \gam_x + r_x (\ri \gam) + (p+\ri q)\gam_{xx} + r (\ri \gam_x), \gam\rangle \\
&{}+ \langle \gam_x, (p+\ri q) \gam_x +r(\ri \gam)\rangle\\
={}& \ri r_x \langle \gam,\gam\rangle + (p+\ri q) \langle \gam_{xx}, \gam\rangle + (p-\ri q) \langle \gam_x, \gam_x\rangle,
\end{align*}
where several terms vanish because $\langle \gam_x, \gam \rangle= 0$. Differentiating the latter gives
$\langle \gam_{xx}, \gam\rangle + \langle \gam_x, \gam_x\rangle=0$. Using this, and $\langle \gam, \gam \rangle=1$, we get
\[
0=\di/\di{t}\bert_{t=0} \bigl\langle \gamf_x, \gamf \bigr\rangle = \ri r_x -2\ri q \langle \gam_x, \gam_x \rangle.
\tag*{\qed}
\]
\renewcommand{\qed}{}
\end{proof}

In particular, the deformations generated by vector fields
$\VH = \ri \gam$ and $\VR_f = f(x) \gam_x$, where $f(x)$ is any $L$-periodic function, preserve the Legendrian condition. These deformations are, respectively, a constant-speed rotation along the fibers of the Clifford map, and reparametrization in $x$ (while fixing the period). We wish to consider functionals on $\Leg_L$ that
are invariant under these transformations (which form a group); accordingly, we define $\Legq_L$ to be the quotient of $\Leg_L$ by the action of this group,
and let $[\gam] \in \Legq_L$ denote the equivalence class of $\gam \in \Leg_L$.

A tangent vector $\qV \in T_{[\gamma]} \Legq_L$ corresponds to an equivalence
class of vector fields $\VV$ along $\gam$, any two of which differ by adding the sum of a constant multiple of $\VH$ and a vector field of the form $\VR_f$.
Thus,
$\qV$
has a unique representative of the form
\[\VV_0 =\dfrac{r_x}{2|\gam_x|^2} (\ri \gam_x) + r (\ri \gam), \]
such that $\int_0^L r\dx =0$.

We define the following skew-symmetric form on $\Legq_L$. Given $\qV, \qW \in T_{[\gamma]}\Legq_L$, choose
representatives $\VV,\VW \in T_\gamma \Leg_L$, and compute
\begin{equation}\label{our2form}
\Omega_{[\gamma]} (\qV, \qW)
= -\int_0^L \det{}_\R( \gam, \gam_x, \VV, \VW)\dx,
\end{equation}
where the notation $\det_\R$ means that we apply a standard\footnote{%
In other words, column $(z_1, z_2)^{\mathsf{T}} \in \C^2$ is identified with column $(\Re z_1, \Im z_1, \Re z_2, \Im z_2)^{\mathsf{T}}\in \R^4$.}
identification $\C^2 \cong \R^4$ to each vector before taking the determinant of the resulting $4\times 4$ matrix.
(Note that the value of $\Omega_{[\gamma]}(\qV, \qW)$
does not depend on the choice of representative.)
If we write
\[\VV = p_V\gam_x + q_V (\ri\gam_x) + r_V (\ri \gam)\]
and similarly for $\VW$, then it is easy to compute that
\[\Omega_{[\gamma]}(\qV, \qW) = \frac12\int_0^L (r_V (r_W)_x - r_W (r_V)_x) \dx = \int_0^L r_V (r_W)_x \dx.\]

\begin{Theorem} $\Omega$ is a symplectic form on $\Legq_L$.
\end{Theorem}
\begin{proof}
The skew-symmetric product is non-degenerate, since
if $\Omega_{[\gamma]}(\qV, \qW) = 0$
for all $\qV \in T_{[\gamma]} \Legq_L$, then $r_W$ is constant, and hence $\VW$ is equivalent to zero.

\newcommand{\ev}{\operatorname{ev}}
\newcommand{\pr}{\operatorname{pr}}
The fact that $\Omega$ is a closed 2-form follows as a special case of the calculus developed by Vizman in~\cite{V11} for a `hat pairing' of differential forms. Let $S$ be a compact oriented $k$-dimensional manifold and $M$ a finite-dimensional manifold, and let $\omega$ be a differential $p$-form on $M$ and $\alpha$ a~differential $q$-form on $S$. Then the pairing
\begin{equation*}
\widehat{\omega \cdot \alpha}:=\dashint_S \ev^* \omega \wedge \pr^* \alpha,
\end{equation*}
defines a differential $(p+q-k)$-form on $\mathcal{F}(S,M)$, the space of smooth functions from $S$ to $M$. Here $\ev\colon S \times \mathcal{F}(S,M) \rightarrow M$ is the evaluation map $\ev(x, f):=f(x)$, $\pr\colon S \times \mathcal{F}(S,M) \rightarrow S$ the projection
 $\pr(x, f):=x$, and $\dashint_S$ denotes the fiber integration. Then by \cite[Theorem 1]{V11},
 \begin{equation}
 \label{dif}
 {\bf{d}}(\widehat{\omega \cdot \alpha})=\widehat{({\bf d} \omega) \cdot \alpha}+(-1)^p \widehat{\omega \cdot {\bf d} \alpha}.
 \end{equation}

We will show that the 2-form $\Omega$ in \eqref{our2form} can be realized as the hat pairing $\widehat\nu =\widehat{\nu \cdot 1}=\dashint_{S^1} {\rm ev}^* \nu$, where $\nu$ is the canonical volume form on $S^3$ and $1$ is the constant function on $S^1$. \big(Note that $\nu$ is the pullback to the sphere of the interior product $\imath_E \mu$, where $\mu$ is the standard volume form on $\R^4$ and $E$ is the Euler vector field $r\di/\di r$.\big)

When $\gam$ is an embedding of the circle into the 3-sphere, $\widehat\nu$ applied to a
given $V$ and $W$ in~$T_\gamma \mathcal{F}\bigl(S^1, S^3\bigr)$ reduces to the so-called {\em transgression map} (see~\cite{V11})
\[
\widehat\nu (V, W)=\int_{S^1}\gamma^*(\imath_W (\imath_V \nu) ).
\]
We compute
\[
\widehat\nu (V, W)=\int_\gamma \imath_W (\imath_V\nu) =\int_0^L \nu (\gam_x, V, W) \dx=\int_0^L \det{}_\R( \gam, \gam_x, V, W) \dx.
\]
The closure of $\Omega$ follows immediately from the hat calculus formula~\eqref{dif}, since
${\bf d} \nu=0$, as $\nu$ is a volume form on $S^3$, and ${\bf d} 1=0$.
\end{proof}

The symplectic form defines a correspondence between functionals on
$\Legq_L$ (Hamiltonians) and vector fields in $T \Legq_L$ in the usual way. Namely, given a smooth functional $\funH\colon \Legq_L \to \R$, the associated Hamiltonian vector field $\VW_\funH$ is defined by the correspondence
\[
\rd \funH_{[\gamma]} (\VV)=\Omega_{[\gamma]}(\VV, \VW_\funH).
\]
In other words, whenever $\gamf(x,t)$ is a one-parameter family of Legendrian curves such that ${\rm d}\gamf/{\rm d}t\vert_{t=0} = \VV$, then
\[\dfrac{{\rm d}}{{\rm d}t}\bigg\vert_{t=0} \funH \left( \gamf(\cdot, t)\right) = \Omega_{[\gamma]}(\VV, \VW_\funH).\]

\begin{Proposition}\label{hamlength} Let $\funA$ denote the total arclength functional on $\Legq_L$, defined by
\[\funA\colon \ [\gam] \mapsto \int_0^L \langle \gam_x ,\gam_x \rangle^{1/2} \dx.\]
Then the Hamiltonian vector field for $4\funA$ is equivalent to
\[
\VW_\funA=\dfrac{ \kk_x}{\langle \gam_x, \gam_x\rangle} (\ri \gam_x) + 2 \kk(\ri \gam).
\]
\end{Proposition}

\begin{proof}
Let $\gam \in \Leg_L$, let $\gamf(x,t)$ be a variation of $\gam$, and let
\begin{equation}\label{VVdef}
\VV = \dfrac{\di \gamf(x,t)}{\di t}\Big|_{t=0} = p_V\gam_x + q_V (\ri\gam_x) + r_V (\ri \gam).
\end{equation}
Using \eqref{generalkform}, we compute
\begin{align}
\frac{\di }{\di t}\bigg|_{t=0} \bigl\langle \gamf_x , \gamf_x \bigr\rangle^{1/2}
& =\bigl\langle \gamf_x , \gamf_x \bigr\rangle^{-1/2} \Re \bigl\langle \gamf_{xt} , \gamf_x \bigr\rangle \bert_{t=0}
= \langle \gam_x, \gam_x \rangle^{-1/2} \Re\langle \VV_x, \gam_x \rangle \nonumber\\[-3pt]
&=\langle \gam_x, \gam_x \rangle^{-1/2} \Re\langle \left((p_V+\ri q_V)_x + \ri r_V\right) \gam_x + (p_V +\ri q_V) \gam_{xx}, \gam_x\rangle \nonumber\\
&=\langle \gam_x, \gam_x \rangle^{-1/2}\Re\left( (p_V)_x \langle \gam_x, \gam_x \rangle + (p_V +\ri q_V) \langle \gam_{xx}, \gam_x\rangle \right).
\label{betadotfirst}
\end{align}
Because $\funA$ is invariant under period-preserving reparametrizations, we may assume that $\gam$ has constant speed $c=\langle \gam_x, \gam_x \rangle^{1/2}$.
Then $\langle \gam_{xx}, \gam_x\rangle = \ri c^3 \kk$ from \eqref{generalkform}, and integrating \eqref{betadotfirst} gives
\[
\rd \funA_{[\gamma]} (\VV) = \int_0^L \bigl(c (p_V)_x -c^2 \kk q_V\bigr)\dx=-\int_0^L c^2 \kk q_V \dx.
\]
Setting this equal to
\[
\Omega_{[\gamma]}(\VV, \VW_\funA)=-\int_0^L r_W (r_V)_x \, {\rm d}x=-\int_0^L 2 r_W \bigl(c^2q_V \bigr)\, {\rm d}x,
\] we get
$r_W=\tfrac12 \kk$. Hence,
\[
q_W = \frac{(r_W)_x}{2 \langle \gam_x, \gam_x\rangle} = \frac{\kk_x}{4 \langle \gam_x, \gam_x\rangle}.\tag*{\qed}
\]\renewcommand{\qed}{}
\end{proof}

\begin{Corollary}\label{nonstretch} Let $\gam \in \Leg^1$ be a unit-speed curve $($not necessarily periodic$)$, let $\gamf(x,t)$ be a~variation of $\gam$ through unit-speed curves, and let $\VV$ be as in \eqref{VVdef}. Then $p_s = \kk q$ and $r_s=2q$.
\end{Corollary}
\begin{proof} If we set $\bigl\langle \gamf_x, \gamf_x\bigr\rangle=1$ identically and $x=s$ in \eqref{betadotfirst}, we obtain
\[0 = (p_V)_s+ \Re( (p_V +\ri q_V) \langle \gam_{ss}, \gam_s\rangle ).\]
Then substituting the expansion \eqref{gamss} for $\gam_{ss}$ gives $p_s = \kk q$, while $r_s=2q$ follows from Lemma~\ref{tanque}.\looseness=1
\end{proof}

We now consider the evolution of the curvature function $\kk$ induced
by vector fields in $T_\gamma \Leg$. (Again, for what follows it is not necessary to assume periodicity.)

\begin{Proposition}\label{kdeforms}
Let $\gam(x,t)$ be a family of regular Legendrian curves and let
\begin{equation}\label{myVVform}
\VV = \dfrac{\di \gam}{\di t} = p\gam_x + q (\ri\gam_x) + r (\ri \gam)
\end{equation}
for real-valued functions $p,q,r$ depending on $x$ and $t$. Let $\speed(x,t) = \langle \gam_x, \gam_x \rangle^{1/2}$ be the speed function.
Then the curvature $\kk$ satisfies the PDE
\begin{equation}\label{kevolvegen}
\dfrac{\di \kk}{\di t} = \dfrac{(\speed q_x)_x}{\speed^2}
+ \dfrac{q}{\speed}\biggl( \dfrac{\speed_x}{\speed}\biggr)_x
+ \kk_x p + \speed \bigl(\kk^2 + 4\bigr) q.
\end{equation}
In particular, if we have a family of unit-speed curves, then $x=s$ and
\begin{equation}\label{kevolves}
\kk_t = q_{ss} + (\kk p)_s + 4q.
\end{equation}
\end{Proposition}
\begin{proof}
From differentiating \eqref{generalkform}, we have
\begin{equation}\label{mykdotform}
\kk_t = \speed^{-3} \Im\left( \langle \VV_{xx}, \gam_x\rangle + \langle \gam_{xx}, \VV_x\rangle\right)
- 3\kk\speed^{-1} \speed_t.
\end{equation}
Recall from \eqref{gamxx} that $\gam_{xx} = - \speed^2 \gam + b \gam_x$ where $b:= \speed^{-1} \speed_x + \ri \speed \kk$.

Next, we differentiate the expansion \eqref{myVVform} twice, substitute into \eqref{mykdotform}, and use the
inner product formulas
\begin{alignat*}{3}
\begin{aligned}
&\langle \gam_{xx},\gam\rangle= -\speed^2, && \langle \gam_{xx}, \gam_{xx}\rangle= \speed^2 \bigl(\speed^2 + b \overline{b}\bigr),&\\
&\langle \gam_{xx}, \gam_x\rangle= \speed^2 b, \qquad && \langle \gam_{xxx}, \gam_x \rangle= \speed^2 \bigl(b_x + b^2 - \speed^2\bigr).&
\end{aligned}
\end{alignat*}
As well, from the proof of Proposition~\ref{hamlength}, we have
\begin{equation}\label{betadotsecond}
\speed_t = \speed^{-1}\Re \langle \VV_x, \gam_x\rangle = (\speed p)_x - \speed^2 k q.
\end{equation}
The expression for $\kk_t$ then follows by substituting these expressions into \eqref{mykdotform} and simplifying. Then \eqref{kevolves} follows by setting $\speed=1$ and using the relation $p_s = \kk q$ from Corollary
\ref{nonstretch}.
\end{proof}

Along a unit-speed curve $\gam \in \Leg^1_L$, the Hamiltonian vector field $\VW_{4\funA}$ from Proposition~\ref{hamlength} has a~representative
\[
\frac12\kk^2\gam_s+\kk_s(\ri \gam_s)+2\kk (\ri \gam),
\]
which satisfies the conditions from Corollary~\ref{nonstretch}, and is thus tangent to $\Leg^1_L$. By Proposition~\ref{kdeforms}, this vector field induces the curvature evolution
\[
\kk_t = \kk_{sss} + \frac12 \bigl(\kk^3\bigr)_s + 4 \kk_s,
\]
which differs by a Galilean transformation from the mKdV equation (with $s$ as spatial variable). In what follows, we will identify an infinite hierarchy
of Hamiltonian vector fields that induce evolution equations for curvature which belong to the mKdV hierarchy.

\subsection{The mKdV hierarchy}
In this subsection, we review the recursive construction of the mKdV hierarchy.
What follows is essentially the same as in Olver \cite{OlverLieGroups}, with some changes in notation; for example,
we use $u$ to denote a scalar function of spatial variable $s$ and time $t$, and let $u_0=u$ and $u_1, u_2, \ldots$
denote its successive $s$-derivatives.
In this notation, the mKdV takes the form
\begin{equation}\label{UmKdV}
u_t = u_3 + \frac32 u^2 u_1.
\end{equation}
On the space $J[u]$ of differential polynomials
in finitely many of the $u_j$, we will need the {\it total $s$-derivative} operator $\Dx$
and the {\it Euler operator} $\Euler$, defined respectively by
\[\Dx \uP = \sum_{j=0}^\infty \uu[j+1] \dfrac{\di \uP}{\di u_j},\qquad
\Euler \uP = \sum_{j=0}^\infty (-\Dx)^{j}\biggl( \dfrac{\di \uP}{\di u_j}\biggr), \qquad \uP \in J[u].
\]

The mKdV hierarchy is an infinite sequence of evolution equations
\[\uu_t = \uM_j, \qquad j\ge 1,\]
where $\uM_1=u_1$ and $\uM_2$ is the right-hand side of \eqref{UmKdV}. It can be
constructed recursively as follows.
There is an infinite sequence $\{\rho_j\}$ of differential polynomials in $\uu$ with $\rho_1 = \tfrac12 \uu^2$ and satisfying
\begin{equation}\label{rhocursion}
\opD \Euler \rho_{j+1} = \opE \Euler \rho_{j},
\end{equation}
where
\[\opD = \Dx, \qquad \opE = \Dx^3 + \Dx \uu \Dx^{-1} \uu \Dx.\]
Then $\uM_j = \opD \Euler \rho_j$ (e.g., $\uM_1 = \uu_1$).
Since $\opD$, $\opE$ are a {\em Hamiltonian pair} of skew-adjoint differential operators (in the sense of \cite[Definition~7.19]{OlverLieGroups}) each member of the mKdV hierarchy is a bi-Hamiltonian system; moreover, each $\rho_j$ is a~conserved density for each of these flows (see~\mbox{\cite[Theorem~7.24]{OlverLieGroups}}).

It follows from \eqref{rhocursion} that the $\uM_j$ satisfy the recurrence $\uM_{j+1} = \opR \uM_j$, where
\[\opR = \opE \circ \opD^{-1} = \Dx^2 + \Dx \uu \Dx^{-1} \uu,\]
while the densities satisfy the recurrence $\Euler \rho_{j+1} = \opR^* \Euler \rho_j$ for $\opR^* = \opD^{-1} \circ \opE$. For example,
\begin{alignat*}{5}
&\uM_1= u_1,\qquad\! && \uM_2= u_3 + \frac32 u^2 u_1,\qquad\! && \uM_3= u_{{5}}+\frac52 u_{{0}}^{2}u_{{3}}
+10\,u_{{0}}u_{{1}}u_{{2}}+\frac52 u_{{1}}^{3}+\frac{15}{8} u_{{1}} u_{{0}}^{4},\quad\! && \ldots & \\
&\rho_1= \frac12 u^2,\qquad\! && \rho_2= -\frac12 u_1^2 + \frac18 u^4,\qquad\! && \rho_3=
\frac12 u_{{2}}^{2}+\frac56 u_{{0}}^{3}u_{{2}}+\frac54 u_{{0}}^{2} u_{{1}
}^{2}+\frac1{16}u_{{0}}^{6},\quad\! && \ldots &
\end{alignat*}
It also follows from $\uM_{j+1} = \bigl(\Dx^2 + \Dx \uu \Dx^{-1} \uu\bigr) \uM_j$ that there are differential polynomials $\uN_j$ such that
$\uu \uM_j = \Dx \uN_j$.

Let $\uL_j = 2 \Euler\rho_j$. The connection between the ingredients of the mKdV hierarchy and deformations of Legendrian curves becomes apparent when we observe that if
\[
\VV = p \gam_s + q (\ri \gam_s) + r (\ri \gam) \in T_\gamma \Leg^1,
\]
then by Corollary~\ref{nonstretch}, $p$, $q$ and $r$ satisfy the same relationships as $N_j$, $M_j$ and $L_j$, respectively~-- provided we replace $\uu$ by $\kk$. Moreover, from \eqref{kevolves} the curvature of a unit-speed curve, under flow by $\VV$, evolves by
\[\kk_t = (\opR + 4)q,\]
where again we replace $u$ by $\kk$ in the expression for $\opR$.
Matters being so, we define the following {\it mKdV-type vector fields}
\[\vecV_j = \uN_j \gam_s + \uM_j (\ri \gam_s) + \uL_j (\ri \gam), \qquad j \ge 1.\]
(From now on, we will take $\uL_j$, $\uM_j$, $\uN_j$, as well as the operators for the mKdV hierarchy, as having~$u$ replaced by $\kk$.)
As noted, these are tangent to the submanifold $\Leg^1$, and it is easy to see that on unit-speed curves they induce curvature evolutions by linear combinations of members of the mKdV hierarchy (see below). However, these vector fields are defined on the larger space $\Leg$, and in the periodic case we will explore their properties in relation to the symplectic structure.

For example, when we restrict our attention to the space $\Leg_L$ of periodic curves, the vector field
\[
\vecV_1 = \frac12 \kk^2 \gam_s + \kk_s (\ri \gam_s) + 2 \kk (\ri \gam)
= \dfrac{\kk^2}{2\langle \gam_x, \gam_x \rangle^{1/2}} \gam_x + \dfrac{\kk_x}{\langle \gam_x, \gam_x \rangle} \ri \gam_x + 2\kk (\ri \gam)
\]
is equivalent to the Hamiltonian vector field from Proposition~\ref{hamlength}. Similarly, we will show below
that the rest of the vector fields $\vecV_j$ represent Hamiltonian vector fields on $\Legq_L$.

\begin{Proposition} Let $\qV_j$ be the vector field on $\Legq_L$ represented by $\vecV_j$. Then $\Omega(\qV_m ,\qV_j) =0$ for $j=1,\ldots, m-1$.
\end{Proposition}
\begin{proof} Let $\gam\in \Leg_L$. By definition of $\Omega_{[\gamma]}$, the statement is equivalent to $\int_0^L \uL_m (\uL_j)_x \,{\rm d}x=0$.
Because this equation is invariant under arbitrary reparametrizations, it
suffices to verify this when $\gam\in \Leg^1_L$ and $x=s$, an arclength parameter. The rest of the argument is essentially
the same as the last step in the proof of \cite[Lemma~7.25]{OlverLieGroups}.
\end{proof}

To show that the $\qV_j$ are Hamiltonian, we need to compute the variations of reparametrization-invariant integrals involving the curvature and its derivatives.

\begin{Lemma}\label{grunge} Let $\gam \in \Leg_L$, let $\gamf(\, \cdot\, ,t)\in \Leg_L$ be a smooth variation of $\gam$, and let $\VV$ be as
in~\eqref{VVdef}, but with the subscript~$V$ omitted from the components~$p$,~$q$,~$r$. Let $\kk_j$ denote the $j$th derivative of curvature with respect to arclength and let $\speed = \langle \gam_x, \gam_x \rangle^{1/2}$. Then for the functional
\[
\funH(\gam) = \int_0^L f(\kk, \kk_1, \ldots, \kk_n) \speed \dx,
\]
we have
\[
\dfrac{{\rm d}}{{\rm d}t}\bigg|_{t=0} \funH(\gamf)
=\int_{0}^{L} (\Euler f)\biggl( \dfrac{(\speed q_x)_x}{\speed}
+ q\biggl( \dfrac{\speed_x}{\speed}\biggr)_x
+ \speed^2(\kk^2 + 4) q\biggr) - (\Fop f) \speed^2 k q \dx,
\]
where
\[
\Euler f = \sum_{m=0}^n \biggl(- \dfrac{\di}{\di s}\biggr)^m \dfrac{\partial f}{\partial \kk_m}, \qquad
\Fop f = f - \sum_{m=0}^n \sum_{j=0}^{m-1}\kk_{m-j}\biggl(- \dfrac{\di}{\di s}\biggr)^j \dfrac{\partial f}{\partial \kk_m}.
\]
\end{Lemma}
\begin{proof}
By straightforward differentiation,
\begin{equation}\label{varbf}
\dfrac{{\rm d}}{{\rm d}t}\bigg|_{t=0} \funH(\gamf)
=\int_0^L \Bigg(\sum_{m=0}^n \dfrac{\di f}{\di \kk_m} \dot\kk_m\Bigg) \speed + f \dot{\speed} \dx,
\end{equation}
where the dot indicates $\di/\di t$.
We wish to express $\dot\kk_m$ in terms of $\dot\kk$. For this purpose, note that
\[\left[\frac{\di}{\di t}, \frac{\di}{\di s}\right] = \left[\frac{\di}{\di t}, \speed^{-1} \frac{\di}{\di x}\right] =
-\speed^{-2} \dot\speed \frac{\di }{\di x} = -\speed^{-1} \dot\speed \frac{\di }{\di s}.\]
Applying this to $\kk_j$, we obtain
\[\dot\kk_{j+1} = \frac{\di}{\di s} \dot\kk_j -\speed^{-1} \dot\speed \kk_{j+1}.\]
By iterating this formula, it follows that
\[\dot\kk_m = \left(\frac{\di}{\di s}\right)^m \dot \kk - \sum_{j=0}^{m-1} \left(\tfrac{\di}{\di s}\right)^j \bigl( \speed^{-1} \dot\speed \kk_{m-j}\bigr).\]
We substitute this in \eqref{varbf}, use ${\rm d}x=\speed^{-1} {\rm d}s$ to apply integration by parts with respect to $s$, and substitute expressions for $\dot{k}$ from~\eqref{kevolvegen} and for $\dot{\beta}$ from~\eqref{betadotsecond}, giving
\begin{align*}
{\rm d}\funH[\VV] &= \int_{x=0}^{x=L} \left[\sum_{m=0}^n \dfrac{\di f}{\di \kk_m}
\left(\left(\frac{\di}{\di s}\right)^m \dot \kk - \sum_{j=0}^{m-1} \left(\frac{\di}{\di s}\right)^j \bigl( \speed^{-1} \dot\speed \kk_{m-j}\bigr)\right)\right] + f \speed^{-1} \dot\speed \ds \\
&= \int_{x=0}^{x=L} \left[ \sum_{m=0}^n \left( \left(-\frac{\di}{\di s}\right)^m \dfrac{\di f}{\di \kk_m} \right) \right] \dot \kk \\
&\hphantom{=} + \left[ f - \sum_{m=0}^n \sum_{j=0}^{m-1} \kk_{m-j} \left(\left(-\frac{\di}{\di s}\right)^j \dfrac{\di f}{\di \kk_m}\right) \right] \speed^{-1} \dot\speed \ds
\\
&= \int_{0}^{L} (\Euler f)\biggl( \dfrac{(\speed q_x)_x}{\speed}
+ q\left( \dfrac{\speed_x}{\speed}\right)_x
+ \speed \kk_x p + \speed^2\bigl(\kk^2 + 4\bigr) q\biggr) + (\Fop f) \bigl( (\speed p)_x - \speed^2 k q\bigr) \dx.
\end{align*}
It is easy to check that
\begin{equation}\label{Eulerfident}
(\Euler f) k_x = (\Fop f)_x,
\end{equation}
and thus the terms in ${\rm d}\funH[\VV]$ involving $p$ make up an exact $x$-derivative of a periodic function, whose integral is zero. This gives the desired expression.
\end{proof}

\begin{Proposition} Each mKdV vector field $\vecV_j$ represents a Hamiltonian vector field on $\Legq_L$.
\end{Proposition}
\begin{proof}
Let $\gam \in \Leg_L$, and let functional $\funH$ and vector field $\VV$ be as in Lemma \ref{grunge}. Because of reparametrization invariance of $\funH$, we can assume that $\gam$ is parametrized with constant speed~$\speed$. Then, using $r_x= 2\speed^2 q$, we have
\[{\rm d}\funH[\VV] = \int_0^L (\Euler f) \bigl( q_{xx} + \tfrac12 r_x \bigl(\kk^2 + 4\bigr)\bigr) - \tfrac12 (\Fop f)k r_x \dx.\]
Applying integration by parts to the second derivative term gives
\[{\rm d}\funH[\VV] =\frac12 \int_0^L \bigl( \beta^{-2} (\Euler f)_{xx} + \bigl(\kk^2+4\bigr) \Euler f - \kk \Fop f\bigr) r_x \dx.\]
Since $(\kk \Euler f -\Fop f)_s = \kk (\Euler f)_s$ from \eqref{Eulerfident}, we can express this in terms of an mKdV recursion operator applied to $\Euler f$:
\[
{\rm d}\funH[\VV] = \frac12 \int_0^L r_x \bigl(\delta^2 + \kk \delta^{-1} \kk \delta + 4\bigr) \Euler f \dx = \frac12 \int_0^L r_x (\opR^* + 4) \Euler f \dx,
\]
where $\delta$ denotes the total $s$-derivative.

Recall that the third `$r$' component of $\VV_j$ is $\uL_j$. Since $\uL_j = 2\Euler \rho_j$ and $\opR^*$ is the recursion operator for the $\uL_j$, then
\[\uL_{j+1} + 4 \uL_j = 2(\opR^* + 4)\Euler \rho_j, \qquad j\ge 1.\]
Using this inductively, we have
\[\uL_{j+1} - (-4)^j \uL_1 = \sum_{k=1}^j (-4)^{j-k} (\uL_{k+1} + 4 \uL_k) = 2 (\opR^* + 4) \sum_{k=1}^j (-4)^{j-k} \Euler \rho_k.\]
In particular, if we set $f = \sum_{k=1}^j (-4)^{j-k}\rho_k$, we see that
\[{\rm d}\funH[\VV] = \Omega\bigl(\vecV_{j+1} -(-4)^j \vecV_1, \VV\bigr).\]
Since it has already been shown that $\vecV_1$ is Hamiltonian, it follows that $\vecV_{n}$ is Hamiltonian for any $n\ge 1$.
\end{proof}

From \eqref{kevolves}, when restricted to unit-speed curves, the Hamiltonian vector field $\vecV_j$ induces the curvature evolution
\begin{equation}\label{kMstep}
\kk_t = (\opR + 4) \uM_j = \uM_{j+1} + 4 \uM_j
\end{equation}
for $j\ge 1$. For the sake of convenience, we define $\vecV_0 = \gam_x$, which induces the translation flow $k_t = k_s$.
If we define
\[
\VZ_n = \sum_{j=0}^n (-4)^{n-j} \vecV_j,
\]
then $\VZ_n$ induces the curvature evolution
\[
\kk_t = (-4)^n \uM_1 + (\opR+4) \Biggl( \sum_{j=1}^n (-4)^{n-j} \uM_j \Biggr),
\]
since $\uM_1 = \kk_s$. Using \eqref{kMstep} and telescoping, this simplifies to
\[\kk_t = (-4)^n \uM_1 +\sum_{j=1}^n (-4)^{n-j} (\uM_{j+1} + 4 \uM_j) = \uM_{n+1}.\]
Thus, $\VZ_n$ induces the $(n+1)$st mKdV equation for curvature.

For later use, we now derive the evolution of the Frenet frame under flow $\VZ_1$.

\begin{Lemma}\label{FrenetEvolution} Let $\Gamf(s,t)$ be the ${\rm U}(2)$-valued Frenet frame for a smooth family of curves $\gamf(s,t)$ parametrized
by arclength $s$, such that
\[\dfrac{\di \gamf}{\di t} = \VZ_1 = \biggl(\frac12 \kh^2 -4\biggr) \gamf_s + \kh_s (\ri \gam_s) + 2\kh (\ri \gam),\]
where $\kh(s,t)$ is the curvature. Then $\Gamf$ satisfies $\dfrac{\di \Gamf}{\di t} = \Gamf P$, where
\begin{equation}\label{Z1flowPform}\renewcommand{\arraystretch}{1.2}
P = \begin{pmatrix} 2\ri \kh & 4 - \tfrac12 \kh^2 +\ri \kh_s \\ \tfrac12 \kh^2 -4 + \ri \kh_s & \ri \bigl(\kh_{ss} + \tfrac12\kh^3 -2\kh\bigr) \end{pmatrix}.
\end{equation}
\end{Lemma}
\begin{proof} Since $\gamf$ is the first column of $\Gamf$, then entries of the first column of $P$ are dictated by the flow of $\gamf$.
Since $\Gamf$ is ${\rm U}(2)$-valued then $P$ takes value in ${\mathfrak u}(2)$, and this determines the upper-right entry. Finally,
since $\Gamf_s = \Gamf U$ where $U=\left( \begin{smallmatrix} 0 & -1 \\ 1 & \ri\kh \end{smallmatrix} \right)$, then equating mixed partials leads to the necessary compatibility condition
\[U_t - P_s - [U,P] = 0,\]
which determines the lower right entry in $P$.
\end{proof}

\section{Stationary curves}\label{sec4}
\def\ki{\kk}
\def\gami{\gam}
\def\kf{\kh}

Let $\gami$ be a periodic unit-speed Legendrian curve {with curvature function $\ki$} and let $\gamf(s,t)$ be its evolution by the flow of ${\bf Z}_n$.
We say that $\gami$ is {\it stationary} if $\gamf(s,t)={\rm A}(t)\gami(s-at)$ for some constant $a\in \R$ and ${\rm A}\colon \R\to {\rm U}(2)$.
If $\gami$ is stationary, then the evolving curvature $\kf(s,t):=\ki(s-at)$ is a~periodic traveling wave solution of the $(n+1)$st mKdV equation.

This section focuses on closed curves with non-constant curvature which are stationary for $\VZ_1$, which we will call {\it $s$-loops}. By substituting $\kf$ into the mKdV equation, we see that $\ki(s)=\kf(s,0)$ is a solution of the third-order ODE
\begin{equation}\label{third}
\ki'''+\frac{3}{2}\ki^2\ki'+a \ki' = 0,\end{equation}
where prime denotes ${\rm d}/{\rm d}s$. Obviously the converse also holds: if $\ki$ is a solution of \eqref{third}, then $\kf(s,t):=\ki(s-at)$ is a traveling wave solution of the mKdV equation.

\begin{Remark} Equation \eqref{third} implies that the Clifford projection of an $s$-loop is a closed elastic curve in $S^2$ (see, e.g., \cite{He} and the literature therein). However, due to the constraint on the rationality of the total curvature, the Legendrian lifts of closed elastica of $S^2$ are, in general, not closed.
\end{Remark}

We will refer to a unit-speed Legendrian curve whose curvature is a non-constant periodic solution of \eqref{third} as an {\it $s$-curve}; thus, while $s$-loops are closed, $s$-curves are not necessarily closed despite having periodic curvature.

The study of $s$-loops is organized in two steps. The first is the analysis of the $s$-curves (obtained from periodic solutions of \eqref{third}) and the second is the derivation of the closure conditions, via integration of the Frenet equations \eqref{FrenetEquation}. We describe this scheme for the general case, suppressing some details, before we apply it to a specific class of stationary curves.

\subsection{The general scheme} By integrating \eqref{third} twice, we obtain the two conservation laws
\begin{subequations}\label{second}
\begin{align}
&\kk''+\frac{1}{2}\kk^3+a\kk+\frac12{b}=0, \label{ksecond}\\
&(\kk')^2+\frac{1}{4}\kk^4+a\kk^2+b\kk +c=0. \label{kfirstorder}
\end{align}
\end{subequations}
Thus, $\kk$ is determined by inverting an Abelian integral along the {\it phase curve} $y^2 + P_{a,b,c}(x)=0$, where $P_{a,b,c}$ is the fourth-degree polynomial
\begin{equation*}
P_{a,b,c}(x)=\frac14 x^4+ax^2+bx+c.
\end{equation*}
As a consequence of the Poincar\'e--Bendixon theorem,
periodic solutions of \eqref{kfirstorder} exist if and only if $P_{a,b,c}$ possesses either
\begin{itemize}\itemsep=0pt
\item four distinct real roots $e_1>e_2>e_3>e_4=-(e_1+e_2+e_3)$, which we will call the {\it dnoidal case}; or
\item two distinct real roots $e_1>e_2$ and two complex conjugate roots $-\frac{1}{2}(e_1+e_2)\pm \ri e_3$ where $e_3>0$, which we will call the {\it cnoidal case}; or
\item three distinct real roots, one them with multiplicity two, which is a degenerate limit of the cnoidal case.
\end{itemize}
In each case, the coefficients $a$, $b$ and $c$ of $P_{a,b,c}$ can be written as functions of real parameters $(e_1, e_2, e_3)$. We will refer to $\be=(e_1, e_2, e_3)$ as the {\it modulus} of~$\kk$.

Assuming that $\gami$ is an $s$-curve, we can arrange (by translating in $s$, and reversing the orientation of $\gami$ if necessary)
that $\kk(0)=e_2$, from which it follows that $\kk'(0)=0$.
With this choice of initial condition, the modulus $\be$ uniquely determines $\kk(s)$.
Moreover, any other $s$-curve with the same modulus is of the form $A \gam_0(s+c)$ for some matrix $A\in {\rm U}(2)$ and constant $c$, and is thus {congruent} to $\gami$.
Consequently, the {\em congruence classes} of $s$-curves of dnoidal type are in one-to-one correspondence with the elements of
\[{\mathcal D}=\bigl\{(e_1,e_2,e_3)\in \R^3\mid e_1>e_2>e_3>-(e_1+e_2+e_3)\bigr\},\]
while the congruence classes of $s$-curves of cnoidal type are in one-to-one correspondence with the elements of{\samepage
\[{\mathcal C}=\bigl\{(e_1,e_2,e_3)\in \R^3\mid e_1>e_2, e_3>0\bigr\}.\]
Sets ${\mathcal C}$ and ${\mathcal D}$ are the {\it moduli spaces} of the $s$-curves of dnoidal and cnoidal type, respectively.}

The curvature of an $s$-curve with modulus $\be=(e_1,e_2,e_3)$ is given by
\begin{alignat}{3}
&\kk(s)= \frac{q_{11}+q_{12}\dn^2(qs,m)}{q_{21}+q_{22}\dn^2(qs,m)}, \qquad&& \text{dnoidal case},& \nonumber\\
&\kk(s)= \frac{q_{11}+q_{12}\cn(qs,m)}{q_{21}+q_{22}\cn(qs,m)}, \qquad && \text{cnoidal case}, &\label{curvaturedn}
\end{alignat}
where $\cn$ and $\dn$ are the Jacobi elliptic functions, $m$ is the square of the Jacobi modulus \cite{Law},
and coefficients $q_{ij}$ and parameters $q$, $m$ are certain functions of $\be$ (see~\cite[formulas~(256.00) and~(259.00)]{BF}).
In the degenerate cnoidal case, $m=0$ and $\cn(q s,m)$ is replaced by $\cos(q s)$; in the sequel, we will not consider this case.

The wavelength (i.e., the least period of the curvature) is
\begin{equation}
\label{wl}\omega = \begin{cases}2{\rm K}(m)/{q}, &\text{dnoidal case,}\\
4{\rm K}(m)/{q}, & \text{cnoidal case,}\end{cases}
\end{equation}
where ${\rm K}(m)$ denotes the complete elliptic integral of the first kind. By integrating the curvature
over its wavelength, we obtain the {\it quantum of total curvature}
\begin{equation}\label{curvtot}\Phi_1(\be):=\frac{1}{2\pi}\int_0^{\omega} \kk \ds={\mathtt A}{\rm K}(m)+{\mathtt B}\Pi(n_1,m),
\end{equation}
which can be expressed in terms of $K(m)$ and the complete elliptic integral of the third kind $\Pi(n_1,m)$,
where ${\mathtt A}$, ${\mathtt B}$ and $n_1$ are functions of $\be$ (see in \cite[formulas~(340.03) and (341.03)]{BF}).
By Proposition~\ref{Maslov}, rationality of the quantum of curvature is a necessary condition for
an $s$-curve to be closed; however, this is not sufficient. In order to deduce sufficient conditions, we next define the conserved momentum and show how explicit $s$-curves can be obtained through integration by quadratures.

{\bf The momentum.}
Let $\mathfrak{h}_0(2)$ be the vector space of traceless $2\times 2$ hermitian matrices. For every modulus $\be$, we consider the map ${\mathtt H}\colon \R\to \mathfrak {h}_0(2)$ defined by
\begin{equation}\label{hamiltonian}
{\mathtt H}=
\begin{pmatrix}
2\kk + \tfrac14 b & \kk' + \ri\bigl(\tfrac12 \kk^2 + a-4\bigr) \\
\kk' - \ri\bigl(\tfrac12 \kk^2 + a-4\bigr) & -2\kk -\tfrac14 b
\end{pmatrix},
\end{equation}
where $\kk$ is as in \eqref{curvaturedn}. From (\ref{second}), it follows that ${\mathtt H}$ has constant eigenvalues $\pm \lambda$, where
\[\lambda=\frac{1}{4}\sqrt{256-128a+16a^2+b^2-16c}>0.\]
\big(Given the expression for $\lambda^2$ in terms of the components of $\be$, it is straightforward to
use calculus to prove that it never vanishes in either the cnoidal or dnoidal cases.\big)
Moreover, equation \eqref{second} implies that
\[
{\mathtt H}'+[U, {\mathtt H}]=0, \quad
\text{where $U = \begin{pmatrix} 0 & -1 \\ 1 & \ri \kk\end{pmatrix}$ as in formula~\eqref{FrenetEquation}.}
\]
It follows that, if $\Gamma$ is the moving frame along an $s$-curve $\gam$ with curvature $\kk$, then
\begin{equation}\label{momentum2}\Gamma  {\mathtt H}  \Gamma^{-1}={\mathfrak m},\end{equation}
where ${\mathfrak m}$ is a constant element of ${\mathfrak h}_0(2)$, called the {\it momentum} of $\gamma$. By multiplying $\gam$ by a matrix in ${\rm U}(2)$ if necessary, we can assume that
\begin{equation}\label{simconf}{\mathfrak m}=\begin{pmatrix} -\lambda &0 \\ \hphantom{-}0 &\lambda\end{pmatrix}.\end{equation}

\subsection{Integrability by quadratures}\label{quadsec}
Let ${\mathtt V}^j\colon \R\to \C^2$ for $j=1,2$ be the periodic maps
\begin{equation}\label{Vdefs}
{\mathtt V}^1 = \begin{pmatrix}{\mathtt H}^1_2 \\ -{\mathtt H}^1_1-\lambda\end{pmatrix},\qquad
{\mathtt V}^2 = \begin{pmatrix}-{\mathtt H}^2_2+\lambda \\ {\mathtt H}^2_1\end{pmatrix}.
\end{equation}

\begin{Proposition}\label{exceptional} For every $s \in \R$, ${\mathtt V}^1(s)$ and ${\mathtt V}^2(s)$ belong to the $\mp\lambda$-eigenspaces of ${\mathtt H}(s)$, respectively. They are everywhere linearly independent if $\gam$ is of dnoidal type, or if $\gam$ is of cnoidal type with {\it generic modulus} $($i.e., $8e_2+b+4\lambda\neq 0)$. In the exceptional $($i.e., non-generic cnoidal$)$ case, ${\mathtt V}^1$ and ${\mathtt V}^2$ vanish at $s=h\omega$ for each $h\in \Z$ $($where $\omega$ is the wavelength in $\eqref{wl})$, and are linearly independent everywhere else.

 \end{Proposition}

 \begin{proof}
 The first part of the statement is a straightforward consequence of (\ref{Vdefs}).
Moreover, if~${\mathtt V}^1$ and~${\mathtt V}^2$ are nonzero then they are
linearly independent as they are eigenvectors for different eigenvalues.

Since $\tV^1$, $\tV^2$ are periodic, we can assume that $s\in [0, \omega)$.
Note that the derivative $\kk'$ vanishes at $s=0$ or $s=\tfrac12 \omega$ and is nonzero elsewhere, so
$\tV^1(s)$, $\tV^2(s)$ are nonzero if $s\ne 0$ and $s\ne \tfrac12 \omega$.
Note also that $\kk(0) = e_2$ and $\kk\bigl(\tfrac12\omega\bigr) = e_1$.
Thus, we have
\[\renewcommand{\arraystretch}{1.1}
 \tV^1(0)=\begin{pmatrix} \tfrac12 {\ri}\bigl(e_2^2+2a-8\bigr) \\ -\frac{1}{4}(8e_2+b)-\lambda\end{pmatrix}, \qquad
 \tV^2(0)=\begin{pmatrix} \frac{1}{4}(8e_2+b)+\lambda \\ \tfrac12 {\ri}\bigl(e_2^2+2a-8\bigr)\end{pmatrix},
\]
and
\[\renewcommand{\arraystretch}{1.1}
\tV^1\bigl(\tfrac12\omega\bigr)=\begin{pmatrix}\tfrac12 \ri \bigl(e_1^2+2a-8\bigr)\\ -\frac{1}{4}(8e_1+b)-\lambda\end{pmatrix},\qquad
\tV^2\bigl(\tfrac12\omega\bigr)=\begin{pmatrix} \frac{1}{4}(8e_1+b)+\lambda\\ \tfrac12 {\ri}\bigl(e_1^2+2a-8\bigr)\end{pmatrix}.
\]
Hence, it suffices to prove that
${\mathtt V}^1$ is nowhere zero in the dnoidal or generic cnoidal cases, and that, in the exceptional case, ${\mathtt V}^1$ vanishes only at $s=0$.

Consider the cnoidal case. In this case, expressing $a$ in terms of $e_1$, $e_2$, $e_3$ gives
\begin{equation}\label{acnexpr}
e_1^2+2a-8=\frac{1}{8}\bigl(5e_1^2-2e_1e_2-3e_2^2+4e_3^2-64\bigr)
\end{equation}
for one half of the imaginary part of the upper entry of $\tV^1\bigl(\tfrac12 \omega\bigr)$. Suppose this vanishes; then
\[e_3=\frac{1}{2}\sqrt{-5e_1^2+2e_1e_2+3e_2^2+64}.\]
Using this equality and writing $\lambda$ and $b$ as a functions of $e_1$ and $e_2$, minus one times the bottom entry of $\tV^1\bigl(\tfrac12 \omega\bigr)$ becomes
\[\frac{1}{4}(b+8e_1)+\lambda = \frac{1}{8}(e_1-e_2)\bigl(16+(e_1+e_2)^2\bigr)>0.\]
Hence, ${\mathtt V}^1\bigl(\tfrac12 \omega\bigr)\neq 0$.
It is clear that if $8e_2+b+4\lambda\neq 0$ then $\tV^1(0) \ne 0$. On the other hand, if $4\lambda = -(8e_2 + b)$, then
squaring both sides and using \eqref{acnexpr} gives $e_2^2+2a-8=0$, so $\tV^1(0)=0$.

Consider now the dnoidal case. In this case, the analogue of \eqref{acnexpr} is
\begin{equation}\label{adnexpr}
e_1^2+2a-8 = \frac12\bigl(e_1^2-e_2^2-e_3^2-e_1 e_2 - e_1 e_3 - e_2 e_3 -16\bigr).
\end{equation}
Suppose that left-hand side \big(which is one half of the imaginary part of the upper entry of $\tV^1\bigl(\tfrac12 \omega\bigr)$\big) vanishes.
Recall that the roots $e_1$, $e_2$, $e_3$, $e_4$ have been ordered so that $e_1>e_2>e_3>e_4$. Since their sum is zero, root $e_1$ is positive. Then, setting the right-hand side of \eqref{adnexpr} to
zero and solving gives
\[e_1=\frac{1}{2}\Bigl(e_2+e_3+\sqrt{64+5e_2^2+6e_2e_3+5e_3^2}\Bigr).\]
Then, expressing $b$ in terms of $e_1$, $e_2$, $e_3$ and using the last equation gives
\[b+8e_1=\frac{1}{4}\bigl(16+(e_2+e_3)^2\bigr)\Bigl(2e_2+2e_3+\sqrt{64+5e_2^2+6e_2e_3+5e_3^2}\Bigr)>0.\]
Hence, since $\lambda$ is positive we have $b+8e_1+4\lambda>0$, and so $\tV^1\bigl(\tfrac12 \omega\bigr) \ne 0$.
Next, suppose the upper entry of $\tV^1(0)$ vanishes. Then \eqref{adnexpr} gives
\[e_1^2-e_2^2+e_3^2+e_1e_2+e_1e_3+e_2e_3+16=0.\]
Solving with respect to $e_1$ and taking into account that $e_1>e_4$, we obtain
\[e_1=\frac{1}{2}\Bigl(-e_2-e_3+\sqrt{5e_2^2-2e_2e_3-3e_3^2-64}\Bigr).\]
Again, expressing $b$ in terms of the moduli and using the last equation gives
\[-\frac{1}{4}(b+8e_2)-\lambda=\frac{1}{16}\bigl(-(e_2-e_3)\bigl((e_2+e_3)^2+16\bigr)-\big|(e_2-e_3)\bigl((e_2+e_3)^2+16\bigr)\big|\bigr).\]
Since $e_2>e_3$ the left hand side is strictly negative, and thus $\tV^1(0)\ne0$.
\end{proof}

Let ${\rm J}=\R$ in the generic case and ${\rm J}=\R\setminus \{h\omega\}_{h\in \Z}$ in the exceptional case. From (\ref{momentum2}), it follows that $\Gamma(s) {\mathtt V}^1(s)$ and $\Gamma(s) {\mathtt V}^2(s)$ are $\mp \lambda$-eigenvectors of ${\mathfrak m}$, for every $s\in {\rm J}$.
Since these vectors are fixed up to scalar multiple,
it follows that there exist smooth functions
$\ell_j\colon {\rm J}\to {\mathbb C}$, $j=1,2$ such that
\begin{equation}\label{if1}\frac{{\rm d}}{{\rm d}s}\bigl(\Gamma {\mathtt V}^j\bigr) =\ell_j \Gamma {\mathtt V}^j,\qquad j=1,2.
\end{equation}

Using (\ref{FrenetEquation}), we rewrite ({\ref{if1}) in the form
\begin{equation}\label{if2}
\frac{{\rm d}}{{\rm d}s}\tV^j
=-(U-\ell_j{\rm Id}_{2\times 2}){\mathtt V}^j,\qquad j = 1,2.
\end{equation}
From (\ref{hamiltonian}) and (\ref{Vdefs}), we can express ${\mathtt V}^j$ and its derivative in terms of $\kk$ and $\kk'$. Solving (\ref{if2}}) for $\ell_{1}$ and $\ell_{2}$, we obtain
\[
\Re \ell_1=\Re \ell_2=\frac{4\kk'}{8\kk+4\lambda+b}, \qquad
\Im \ell_1=\frac{3\kk}{4}+\Lambda,\qquad
\Im \ell_2=\frac{\kk}{4}-\Lambda,
\]
where
\begin{equation}\label{Lambda}
\Lambda = \frac{16(4-a)+(b+4\lambda)\kk}{4(8\kk+4\lambda+b)}.
\end{equation}

On the other hand, (\ref{if1}) implies that
\begin{equation*}
\Gamma|_{{\rm J}}\, {\rm e}^{-{\int \! \ell_j {\rm d}s}} {\mathtt V}^j= \Xi^j,\qquad j=1,2,
\end{equation*}
where $\Xi^1$ and $\Xi^2$ are locally constant maps with values in the $\mp\lambda$-eigenspaces of ${\mathfrak m}$.
Since ${\mathfrak m}$ is in diagonal form, $\Xi^1=\bigl(\xi_1^1,0\bigr)^{\mathsf{T}}$ and $\Xi^2=\bigl(0,\xi_2^2\bigr)^{\mathsf{T}}$, where $\xi_1^1$, $\xi_2^2$ are nonzero constants. Choosing appropriately the primitives of the $\ell_j$-functions, we may assume $\tomedit{\xi^1_1}=\xi^2_2=1$. Hence, the Frenet frame along an $s$-curve with modulus $\be$ is given by
\begin{equation}\label{quadrature}
\Gamma|_{{\rm J}}=\sqrt{8\kk+4\lambda+b}
 \begin{pmatrix} {\rm e}^{\ri \int\frac{3}{4}\kk +\Lambda {\rm d}s} & 0 \\ 0 & {\rm e}^{\ri \int \frac{1}{4}\kk -\Lambda {\rm d}s}\end{pmatrix} {\mathtt V}^{-1},
\end{equation}
where ${\mathtt V}=\bigl({\mathtt V}^1,{\mathtt V}^2\bigr)$ and $\Lambda$ is as in \eqref{Lambda}.
Using (\ref{curvaturedn}), we can rewrite $\Lambda$ in the form
\begin{alignat*}{3}
&\Lambda(s)=
 \frac{p_{11}+p_{12}\dn^2(qs,m)}{p_{21}+p_{22}\dn^2(qs,m)},\qquad && \text{dnoidal case,}& \\
&\Lambda(s)= \frac{p_{11}+p_{12}\cn(qs,m)}{p_{21}+p_{22}\cn(qs,m)},\qquad && \text{cnoidal case},&
\end{alignat*}
the coefficients $p_{ij}$ being appropriate functions of $\be$. Using standard elliptic integrals we obtain
\[
\Phi_2(\be):=\frac{1}{2\pi}\int_0^{\omega} \Lambda \ds={\mathtt C}{\rm K}(m)+{\mathtt D}\Pi(n_2,m),
\]
where ${\mathtt C}$, ${\mathtt D}$ and $n_2$ are functions of the modulus~$\be$. As a consequence of \eqref{curvtot} and \eqref{quadrature}, we have the following.

\begin{Proposition} An $s$-curve with modulus $\be$ is closed if and only if
\[(\Phi_1(\be),\Phi_2(\be))\in {\mathbb Q}^2.\]\end{Proposition}

\begin{Remark}
If we compute the {\it monodromy} ${\mathtt M}:=\Gamma|_{\omega}\, \Gamma|_0^{-1}\in {\rm U}(2)$ for an $s$-curve $\gam$, then $\gam$ is closed if and only if ${\mathtt M}$ has finite order, the order
being the wave number of $\gam$.\end{Remark}

\begin{Proposition} \label{timeevolution}Let $\gami$ be an $s$-curve with modulus $\be$ and momentum ${\mathfrak m}$ given by \eqref{simconf}.
Then
\begin{equation*}
\gamf(s,t)= \exp \biggl(\ri \biggl({\mathfrak m}-\frac14 b{\rm Id}_{2\times 2}\biggr)t\biggr) \gami(s-at)
\end{equation*}
is the evolution of $\gami$ by the flow of $\VZ_1$. \end{Proposition}

\begin{proof} From \eqref{FrenetEquation} and Lemma \ref{FrenetEvolution}, we have
\[\Gamf^{-1}  {\rm d}\Gamf = {\rm U}(s-at)\ds + P(s-at)\, {\rm d}t,\]
where ${\rm U}(s-at) = \left( \begin{smallmatrix} 0 & -1 \\ 1 & \ri k(s-at) \end{smallmatrix} \right)$ and $
P(s-at)$ is given by \eqref{Z1flowPform} with $\kh(s,t) = k(s-at)$.
Using~\eqref{ksecond}, we can write
\[\kk_{ss}(s-at)+\frac12 \kk(s-at)^3-2\kk(s-at)=-(a+2)\kk(s-at)-\frac12 b,\]
so that
\begin{equation}\label{Pstationary}
P(s-a t) =
\begin{pmatrix} 2\ri \kh & 4-\tfrac12 \kh^2 +\ri \kh_s \\
-4+\tfrac12 \kh^2 + \ri \kh_s & -\ri\bigl((a+2)\kh+\tfrac12{b}\bigr) \end{pmatrix}.
\end{equation}

Since $\gami$ is stationary, its evolving Frenet frame satisfies $\Gamf(s,t)=\rA(t) \Gamma(s-at)$, where
${\rA\colon\R\to {\rm U}(2)}$ is a smooth map with $\rA(0)={\rm Id}_{2\times 2}$.
Substituting this into $\frac{\partial \Gamf}{\partial t} = \Gamf P(s-at)$, we deduce that
\begin{align}
P(s-at)={}& \Gamma(s-at)^{-1} \rA(t)^{-1}\frac{\di}{\di t} (\rA(t) \Gamma(s-at))\nonumber\\
={}&\Gamma(s-at)^{-1} \rA(t)^{-1}\left(\rA'(t)  \Gamma(s-at)+\rA(t) \frac{\di}{\di t} \Gamma(s-at)\right)\nonumber\\
={}&\Gamma(s-at)^{-1} \rA(t)^{-1} \rA'(t) \Gamma(s-at) -a {\rm U}(s-at).\label{eqU2}
\end{align}
From \eqref{Pstationary} and \eqref{hamiltonian}, we obtain
\begin{equation}\label{eqUU2} P(s-at)+a {\rm U}(s-at)=\ri \left({\mathtt H}(s-at)-\frac14 b\,{\rm Id}_{2\times 2}\right).\end{equation}
Then \eqref{eqU2} and \eqref{eqUU2} imply that
\[ \rA(t)^{-1} \rA'(t) =\ri \Gamma(s-at) {\mathtt H}(s-at) \Gamma(s-a t)^{-1}
-\ri \frac14 b\,{\rm Id}_{2\times 2}= \ri \left({\mathfrak m}-\frac14{b}\,{\rm Id}_{2\times 2}\right).\]
Exponentiating the matrix on the right and substituting in $\gamf(s,t) = \rA(t) \gami(s-a t)$ gives the required formula.
\end{proof}

\begin{Remark} \label{timeevolution2} The evolution of $\gami$ is periodic in $t$ if and only if there exist $p_1,p_2\in {\mathbb Q}$ such that $b=p_1\lambda$ and
 $\pi a= p_2\omega\lambda$. An $s$-loop satisfying these two constraints is said to be {\it $t$-periodic}. However, in some cases (for instance if $b=0$) the trace $|\gamf(-,t)|$ can evolve periodically in time even though $\gami$ is not time-periodic.
\end{Remark}

\subsection[Phase-symmetrical s-curve of cnoidal type]{Phase-symmetrical $\boldsymbol{s}$-curve of cnoidal type}

We now specialize the general scheme to $s$-curves and $s$-loops of \emph{cnoidal type} whose phase curves are symmetrical about the imaginary axis; we call these {\it $\phi$-curves and $\phi$-loops} for short.
The moduli of a $\phi$-curve satisfy
$e_2=-e_1$. We drop the dependence upon $e_2$ and we identify the moduli space
of $\phi$-curves with the quadrant
${\mathcal C}_0=\bigl\{(e_1,e_3)\in \R^2 \mid e_1>0, \, e_3>0\bigr\}$.
(The zero subscript indicates our symmetric assumption.)
The curvature and the wavelength of a $\phi$-curve with modulus $\be=(e_1,e_3)$ are given by
\begin{equation}\label{curvaturecndn}
\kk =
-e_1\cn(\ell s,m),\qquad
\omega_\be=
\frac{4}{\ell}{\rm K}(m),
\end{equation}
where $m=\frac{e_1^2}{e_1^2+e_3^2}$ and $\ell = \frac12 |\be|=\frac12 \sqrt{e_1^2 + e_3^2}$.
From the curvature formula it follows that the quantum $\Phi_1$ of total curvature of a $\phi$-curve vanishes. As a consequence we have:

\begin{Corollary} A $\phi$-curve of modulus $\be$ is closed if and only if $\Phi_2(\be)\in {\mathbb Q}$. The Maslov index of a $\phi$-loop is zero.
\end{Corollary}

\begin{Remark} The modulus $\be$ is exceptional (in the sense of Proposition \ref{exceptional}) if and only if it is an element of ${\mathcal C}^{*}_0=\{\be\in {\mathcal C}_0 \mid |\be|=4\}$.
\end{Remark}

In order to compute $\Phi_2(\be):=\frac{1}{2\pi}\int_0^{\omega_\be} \Lambda \, {\rm d}s$, we start with the expressions for $a$ and $\lambda$ in terms of the modulus, take into account that $b=0$, and use (\ref{curvaturecndn}). Then expression~\eqref{Lambda} for the $\Lambda$-function becomes
\begin{equation}\label{Lambda2} \Lambda = \dfrac{ 16+e_1^2-e_3^2 -\lambda e_1 \cn(\ell s, m)}{4\lambda - 8 e_1 \cn(\ell s, m)},
\end{equation}
while \eqref{Lambda} specializes to $\lambda = \tfrac14 \sqrt{(e_1^2+e_3^2-16)^2 + 64 e_1^2}$.

From (\ref{Lambda2}), using the standard elliptic integral (341.03) of \cite{BF}, we obtain
\[
\Phi_2(\be)=\begin{cases}\displaystyle \frac{\lambda}{2\pi |\be|}\biggl({\rm K}(m)-
\frac{|\be |^2+16}{|\be |^2-16}\Pi \biggl(\frac{-64e_1^2}{(|\be |^2-16)^2},m\biggr)\biggr), &\be\notin {\mathcal C}_0^*,\\
\displaystyle \frac{e_1}{4\pi}{\rm K}\biggl(\frac{e_1^2}{16}\biggr), & \be\in {\mathcal C}_0^*.
\end{cases}
\]
The function $\Phi_2$ is real-analytic on ${\mathcal C}_0\setminus {\mathcal C}^*_0$ and has a jump discontinuity on ${\mathcal C}^*_0$. However, its regularization
\[
\Phit_2(\be)=\begin{cases}\displaystyle \Phi_2(\be), & \be\in {\mathcal C}^-_0,\\
\displaystyle\Phi_2(\be)+\frac{1}{2}, & \be\in {\mathcal C}^*_0,\\
\displaystyle\Phi_2(\be)+1, & \be\in {\mathcal C}^+_0,
 \end{cases}
 \]
 where ${\mathcal C}^{+}_0=\{\be\in {\mathcal C}_0 \mid |\be|>4\}$ and ${\mathcal C}^{-}_0=\{\be\in {\mathcal C}_0 \mid |\be|<4\}$ are the connected components of~${\mathcal C}_0\setminus {\mathcal C}^*_0$, is real analytic on~${\mathcal C}_0$.
 Since $\Phit_2$ differs from $\Phi_2$ by a rational number, we have the following.

 \begin{Corollary} A $\phi$-curve with modulus $\be$ is closed if and only if
 $\Phit_2(\be)\in {\mathbb Q}$.\end{Corollary}

 \begin{Definition} Let $\gam$ be a $\phi$-loop with modulus $\be$ such that $\Phit_2(\be)=q\in {\mathbb Q}$. We call $q$
 the {\it characteristic number} of $\gam$. \end{Definition}

An experimental analysis of the symplectic gradient of $\Phit_2$ leads to the following observations (see also Figure \ref{FIG1S4New}):
\begin{itemize}\itemsep=0pt
\item The map $\Phit_2$ is a submersion from ${\mathcal C}_0$ onto the interval ${\mathtt J}_{1/2}=(1/2,+\infty)$. For $r\in \tJ_{1/2}$, let~$\Sigma_r$ denote the {\it fiber} (or level set) $\Phit_2^{-1}(r)$. Then $\Sigma_r$ intersects ${\mathcal C}^*_0$ transversely in a~single point.
In particular, for every $q\in {\mathtt J}_{1/2}\cap {\mathbb Q}$, the points along $\Sigma_q$ correspond to a~$1$-parameter family of distinct congruence classes of $\phi$-loops with characteristic number~$q$.

\item The fiber $\Sigma_1$ is unbounded, and $\partial \Sigma_1=\{\be_1^{-}\}$ where $\be_1^{-}=(0,2)$. When $\be \in \Sigma_1$ tends to $\be_1^{-}$, the tangent to $\Sigma_1$ at $\be$ limits to a horizontal line. In addition, the oblique line $e_3=p e_1$ for $p \approx 0.448103$ is an asymptote of $\Sigma_1$.

\item If $r\in (1/2,1)$, the fiber $\Sigma_r$ is bounded, lies above $\Sigma_1$, and has
 two boundary points along the $e_3$-axis: the upper boundary point $\be_r^{+}=\bigl(0,e_{3,r}^{+}\bigr)$ where $e_{3,r}^{+}>4$, and the lower boundary point $\be_r^{-}=(0,e_{3,r}^{-})$ where $2< e_{3,r}^{-}<4$. When $\be \in \Sigma_r$ tends to $\be_r^{\pm}$, the tangent to $\Sigma_r$ at $\be$ tends to a horizontal line.

\item If $r>1$, the fiber $\Sigma_r$
 lies below $\Sigma_1$ and has a unique boundary point $\be_r^{-}=(0,e_{3,r}^{-})$,
 with $e_{3,r}^{-}\in (0,2)$. As in the previous case, when $\be \in \Sigma_r$ tends to $\be_r^{-}$, the tangent to $\Sigma_r$ at $\be$ tends to the horizontal line $e_3=e_{3,r}^{-}$. In addition, the $e_1$-axis is an asymptote of $\Sigma_r$.
 \end{itemize}

\begin{figure}[t]\centering
\includegraphics[height=8cm,width=8cm]{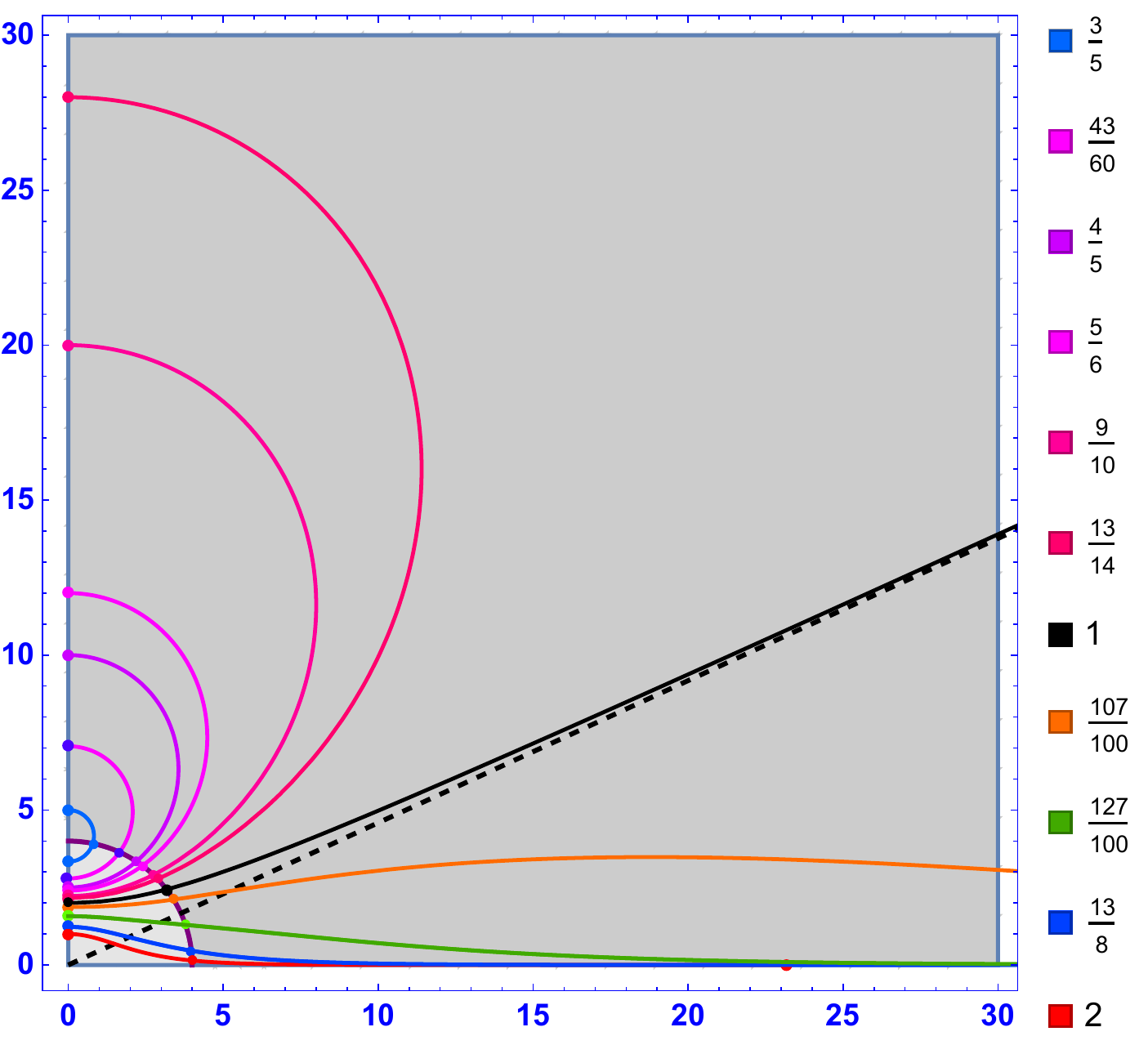}
\caption{Selected level sets $\Sigma_r$ for $r$-values specified in the legend.}\label{FIG1S4New}
\end{figure}

For $q\in {\mathtt J}_{1/2}\cap {\mathbb Q}$, we call $\Sigma_q$ the {\it modular curve of $q$}. The exceptional point $\Sigma_q\cap {\mathcal C}^*_0$ is denoted by $\be_q^*$.
Then $\Sigma_q\setminus\{\be_q^*\}$ has two connected components: $\Sigma^-_q\subset {\mathcal C}^-_0$ and $\Sigma^+_q\subset {\mathcal C}_0^+$.

\subsection*{Congruence class representative} For a given modulus
$\be$, we pick a representative $\phi$-loop $\gam_\be$, uniquely specified
by initial conditions
\[\gam_\be(0)=\|\tU_{\be}^{-1}\|\tU_\be, \qquad
\gam'_\be(0)=\|\tU_{\be}^{-1}\| \tU^*_{\be},\]
where
\[\tU_{\be}=\begin{pmatrix} -4\ri(2e_1+\lambda)\\ |\be|^2-16\end{pmatrix}, \qquad \tU^*_{\be}=\begin{pmatrix}-\overline{\tU^2}_{\be}\\
\overline{\tU^1}_{\be}\end{pmatrix}.\]
We call this is a {\it standard $\phi$-loop}; clearly, any $\phi$-loop is congruent to a standard one.
From now on we assume that $\phi$-loops are in their standard forms.

\begin{Definition}
We call the $1$-parameter family ${\mathcal G}_q=\{\gam_{\be}\mid \be\in \Sigma_q\}$ of $\phi$-loops the {\it isomonodromic family} of $q$. The terminology is motivated by the fact that every loop in ${\mathcal G}_q$ has monodromy
\[{\mathtt M}_q= \begin{pmatrix} {\rm e}^{\ri 2\pi\,q} & 0\\
0& {\rm e}^{-\ri 2\pi\, q}\end{pmatrix}.
\]
(The diagonal form of the monodromy is the reason for our choice of initial conditions for standard loops.)
\end{Definition}

 \begin{figure}[t]\centering
\includegraphics[height=5cm,width=5cm]{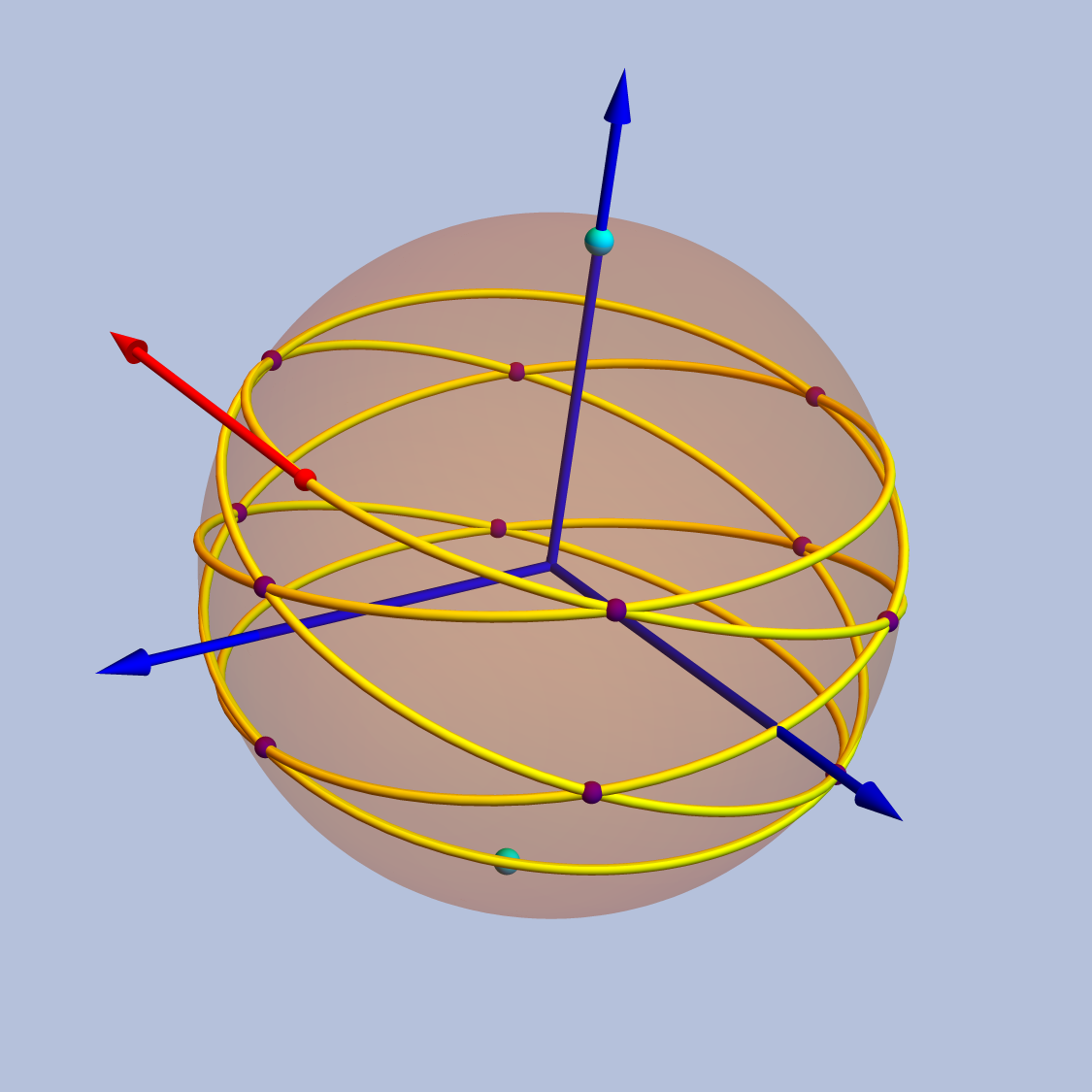}\qquad
\includegraphics[height=5cm,width=5cm]{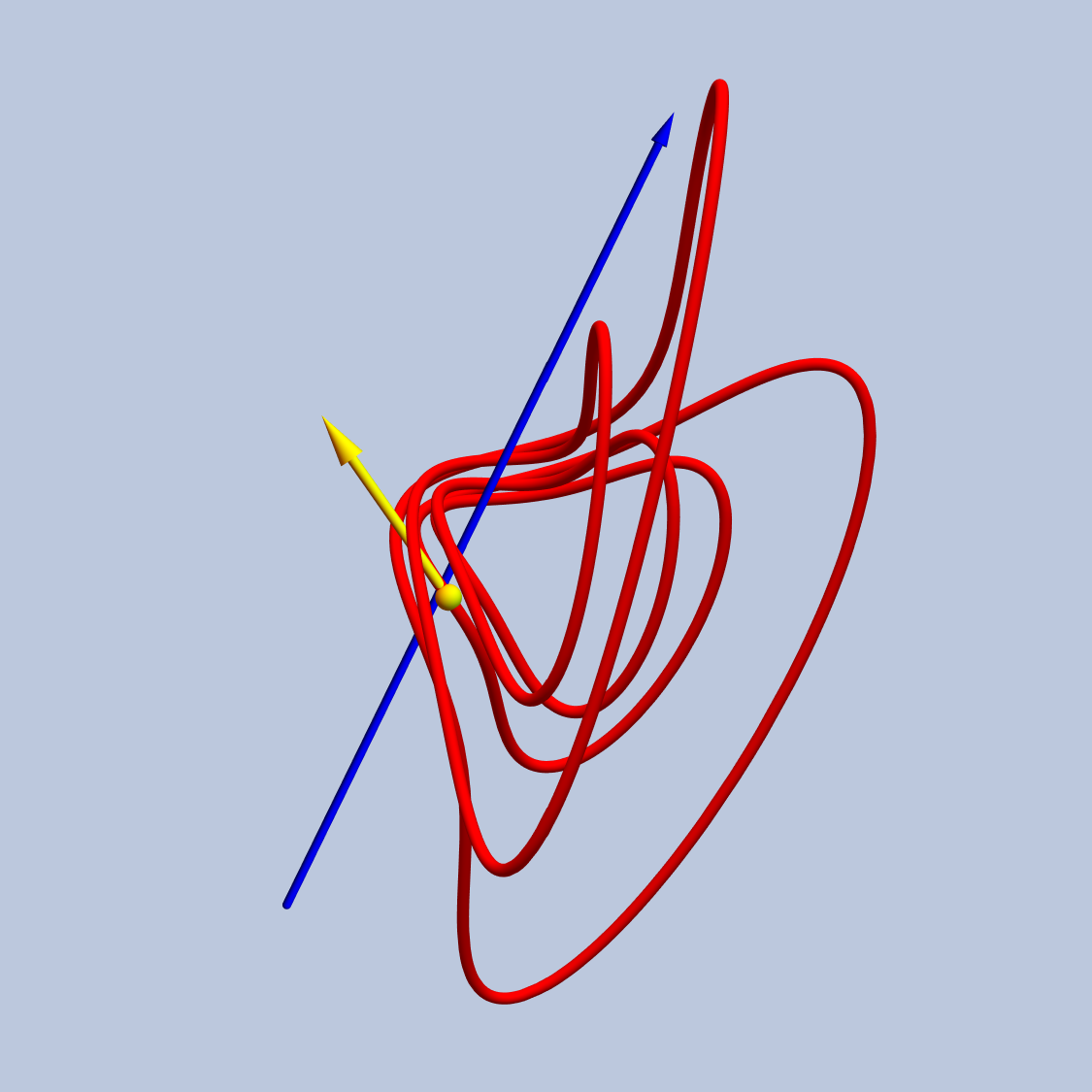}
\caption{Left: for $\be=(0.600642,2.44722)\in \Sigma_{5/6}$, the Clifford projection $\eta_{\be}$ of the loop $\gam_\be$; the blue vector passing through the cyan point is ${\bf i}=(1,0,0)$. Right: the Heisenberg projection $(\gam_{\be})_H$ of $\gam_{\be}$; the blue vector is parallel to ${\bf k}=(0,0,1)$.}\label{FIG3S4New}
\end{figure}

 \begin{figure}[t]\centering
\includegraphics[height=5cm,width=5cm]{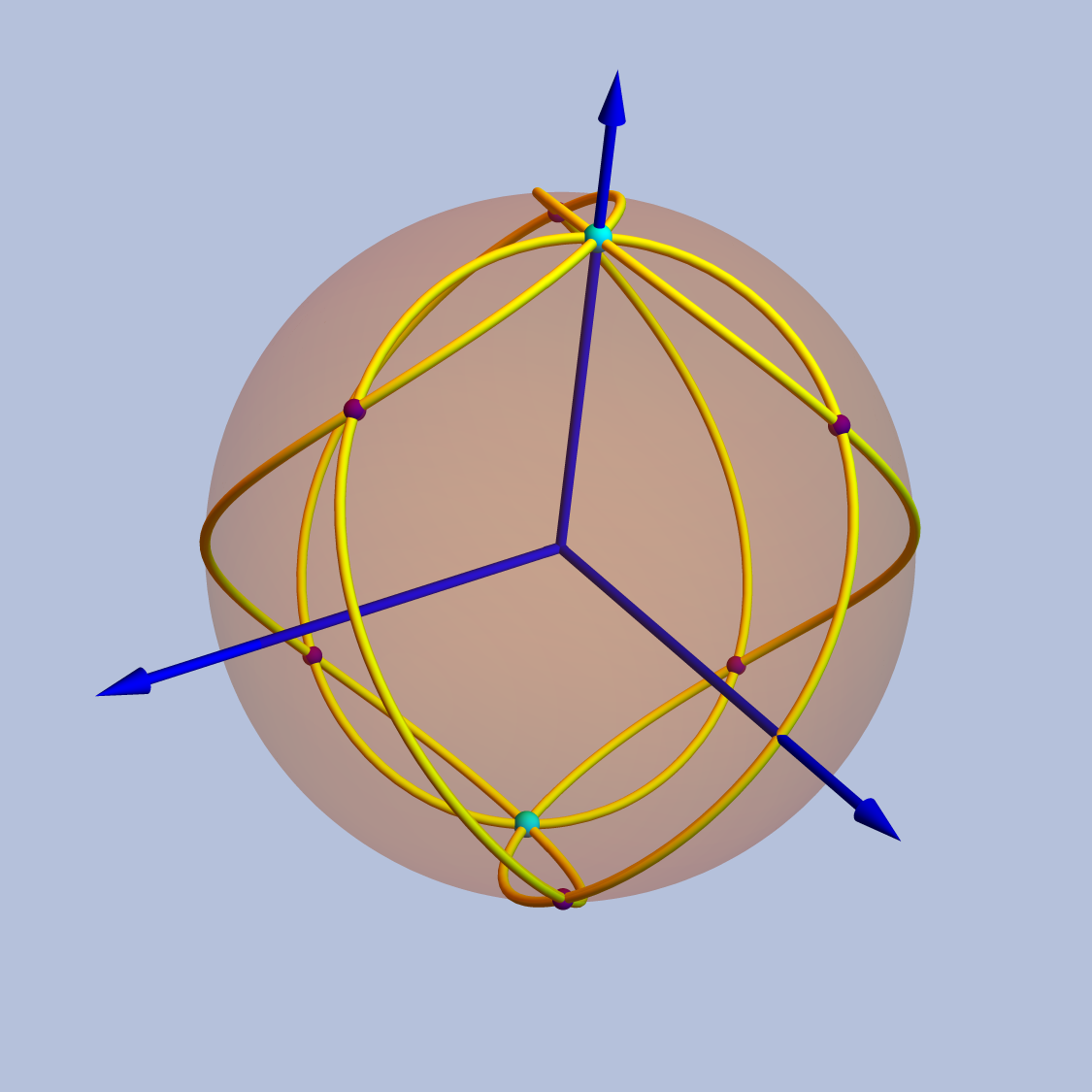}\qquad
\includegraphics[height=5cm,width=5cm]{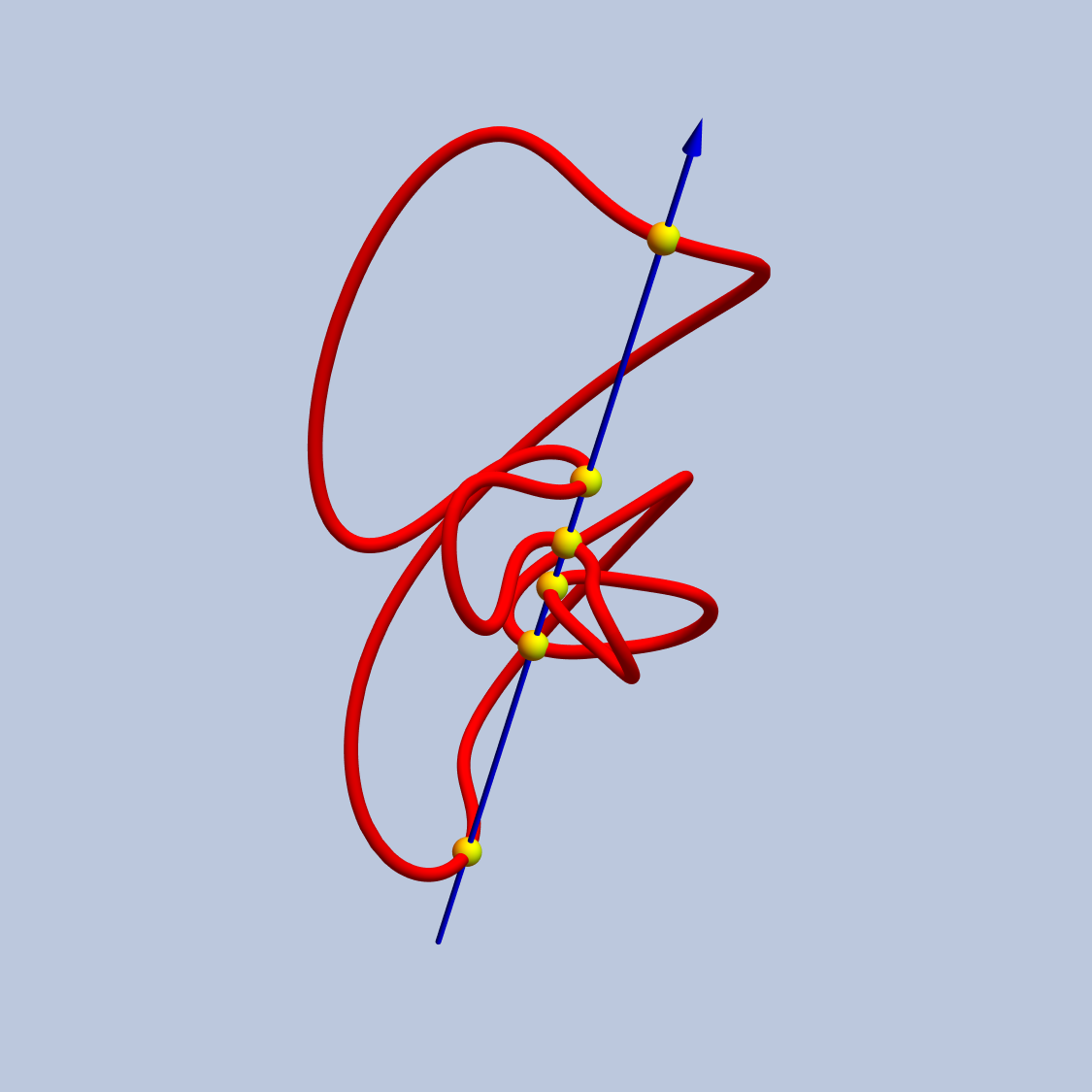}
\caption{Left: for $\be^*_{5/6} \approx (2.39412, 3.2044)$, the Clifford projection $\eta_{\be}$ of the loop $\gam_\be$. Right: the Heisenberg projection $(\gam_{\be_{5/6}^*})_H$. (Blue vectors are the same as in Figure \ref{FIG3S4New}.)}\label{FIG4S4New}
\end{figure}

Finally, we discuss the geometry of $\phi$-loops, along with several examples.

\subsection*{Geometric features} Let $\be\in \Sigma_q$ for $q=m/n\in {\mathtt J}_{1/2}\cap {\mathbb Q}$. From formula~\eqref{quadrature} for the Frenet frame, setting $\Phi_1=0$, $\Phit_2=q$ along $\Sigma_q$, and using the properties of the map $\Phit_2$, we derive the following results:
\begin{itemize}\itemsep=0pt
\item If $n$ is odd, the spin $\rs$ and the Clifford index $\cl$ of $\gam_{\be}$ are both $1$. If $n$ is even, $\rs=1/2$ and~${\cl=2}$.

\item The image of $\gam_{\be}$ is invariant by the group of order $n$ generated by the monodromy, while the image of the Clifford projection $\eta_{\be}$ is invariant by the group generated by the rotation of $2\pi/(\rs n)$ around the $x$-axis.

\item The image of $\eta_{\be^*_q}$ passes through each of the poles ${\rm N}_{\pm}=(\pm1,0,0)$ $n \rs$ times. If $\be\neq \be^*_q$, the image of $\eta_{\be}$ is bounded by the planes
$x=\eta^1_{\be}(0)\in (0,1)$ and $x=\eta^1_{\be}(\omega_{\be}/2)=-\eta^1_{\be}(0)\in (-1,0)$, where $\eta^1_\be$ denotes the first
component of $\eta_{\be}$.

\begin{figure}[t]\centering
\includegraphics[height=5.4cm]{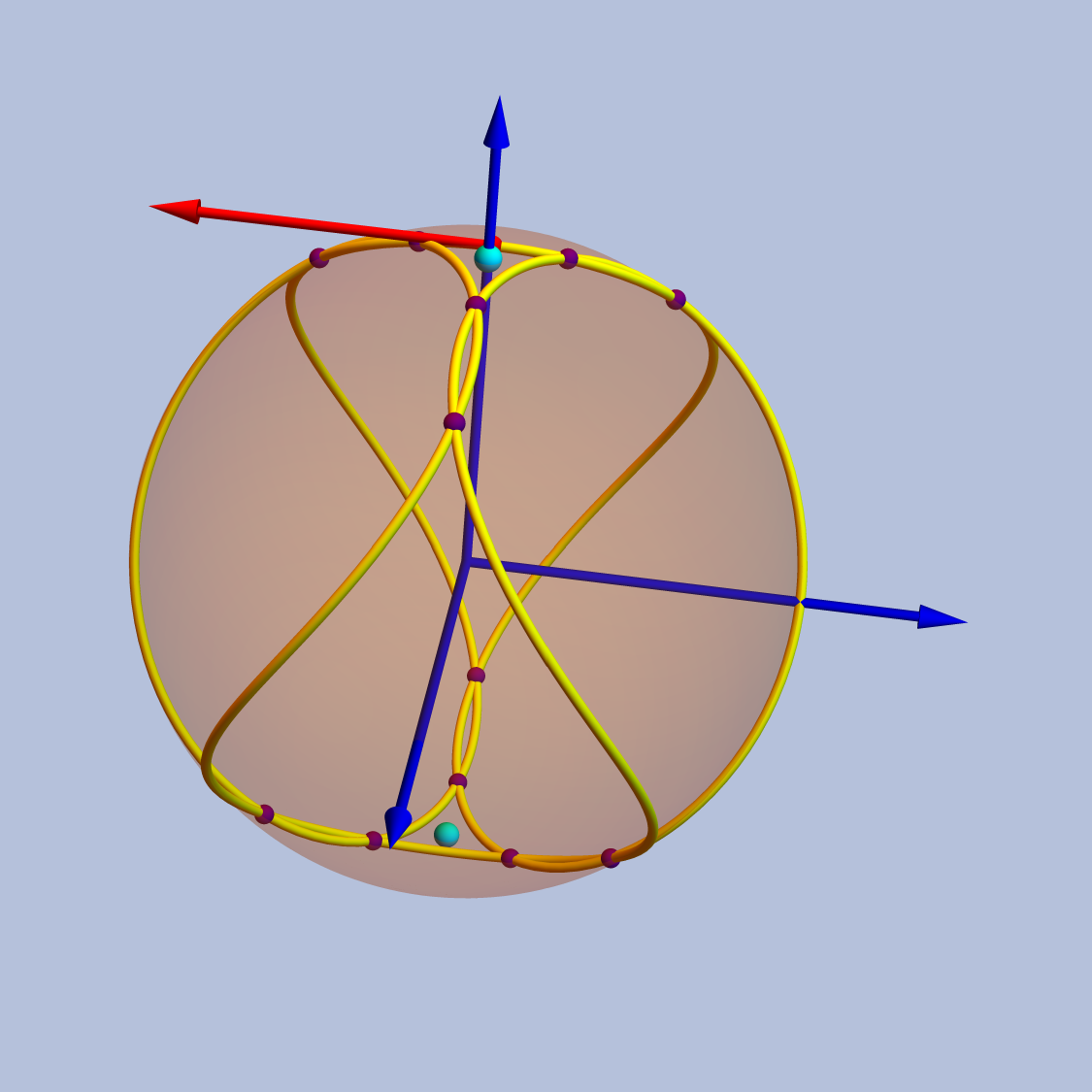}\qquad
\includegraphics[height=5.4cm]{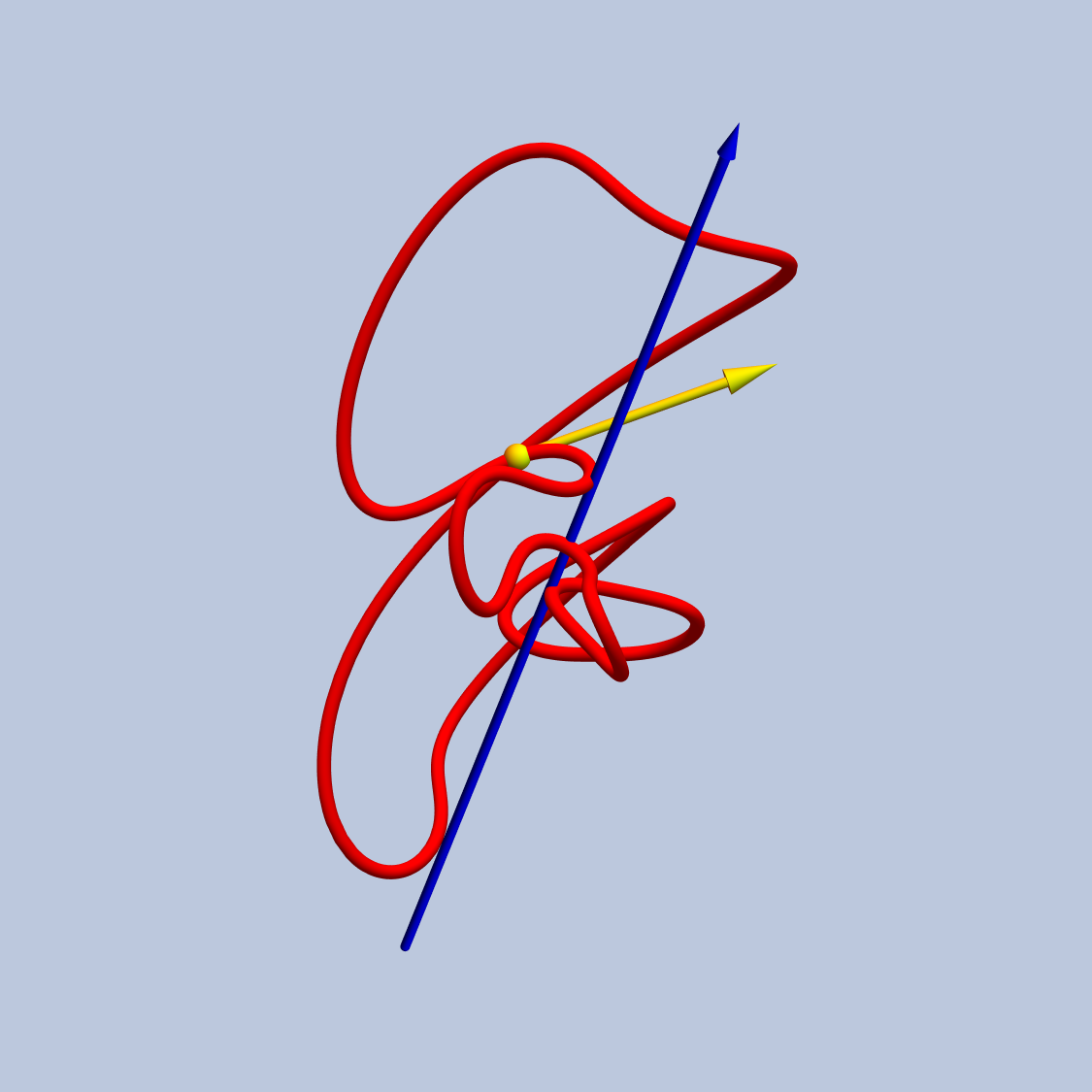}
\caption{\small{Left: the Clifford projection $\eta_{\be}$ of the loop with modulus $\be=(2.59382, 3.35732)\in \Sigma_{5/6}$ The blue vector passing through the cyan point is ${\bf i}=(1,0,0)$. Right: the Heisenberg projection $(\gam_{\be})_H$ of $\gam_{\be}$. The blue vector is parallel to ${\bf k}=(1,0,0)$.}}
\end{figure}

\item Let ${\mathtt S}^1_x\subset \rS^2$ be the equator $\rS^2 \cap Oyz$
oriented counterclockwise with respect to ${{\bf i}=(1,0,0)}$. Taking the homotopy class of ${\mathtt S}^1_x$ as generator, we identify the fundamental group $\pi_1\bigl(\rS^2\setminus \{{\rm N}_{\pm}\}\bigr)$ with $\Z$.
Similarly, if we let ${\mathtt S}^1_z\subset \R^3\setminus Oz$ be the unit circle centered at the origin and contained in the plane $z=0$, equipped with the counterclockwise orientation with respect to upward oriented $z$-axis,
then that generator allows us to identify the fundamental group $\pi_1\bigl(\R^3\setminus Oz\bigr)$ with $\Z$.
Using these identifications, for a given $\be \in \Sigma^+_q$ the homotopy class $[\eta_{\be}]\in \pi_1\bigl(\rS^2\setminus \{{\rm N}_{\pm}\}\bigr)$ is $2\rs (n-m)$; equivalently, the homotopy class $\big[\widetilde{\gam}_{\be}\big]\in \pi_1\bigl(\R^3\setminus Oz\bigr)$ is $2\rs (n-m)$.
For $\be\in \Sigma^-_q$, the homotopy classes $[\eta_{\be}]$ and $\big[\widetilde{\gam}_{\be}\big]$ are both $2 \rs m\operatorname{sign}(n-m)$.

\begin{figure}[t]\centering
\includegraphics[height=5.4cm]{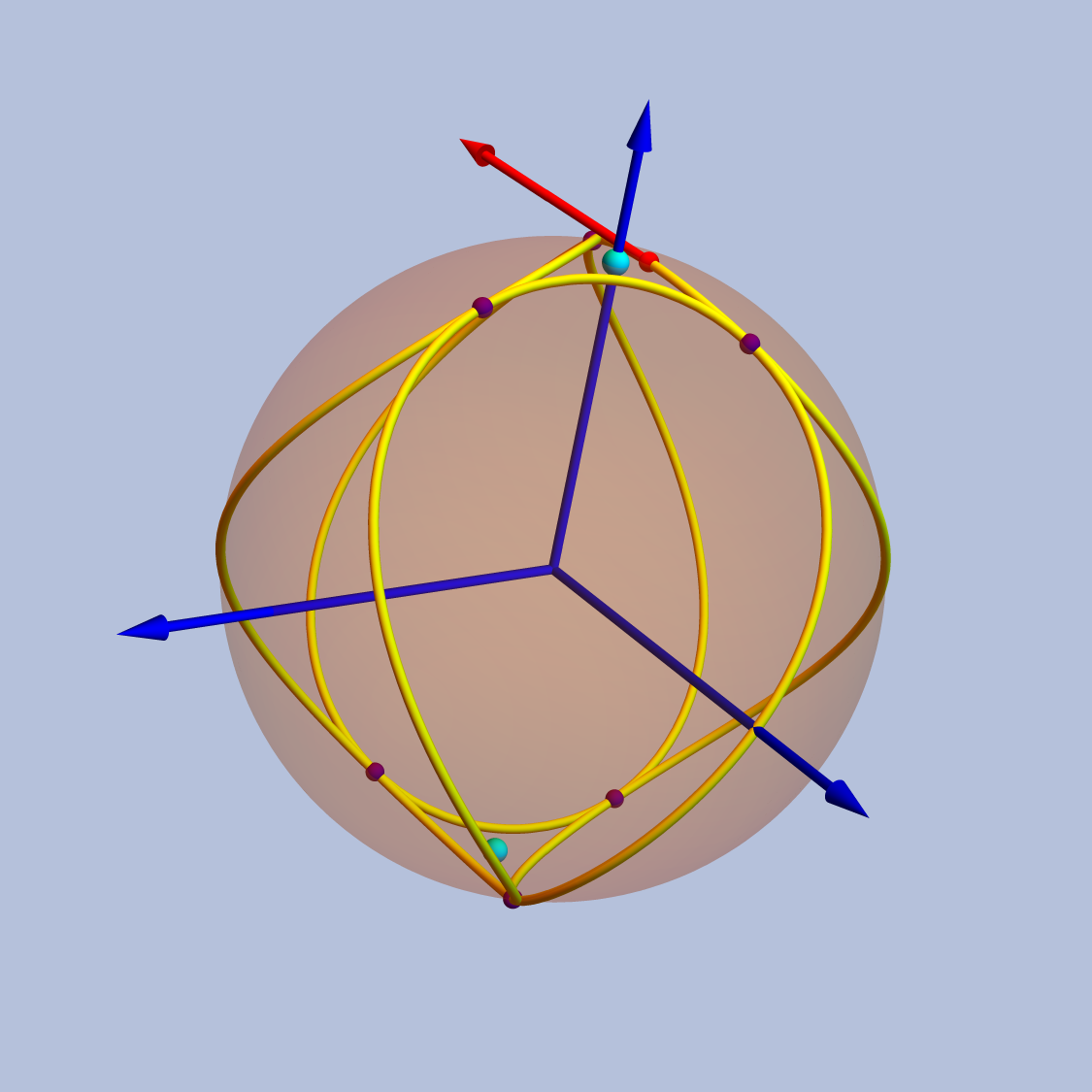}\qquad
\includegraphics[height=5.4cm]{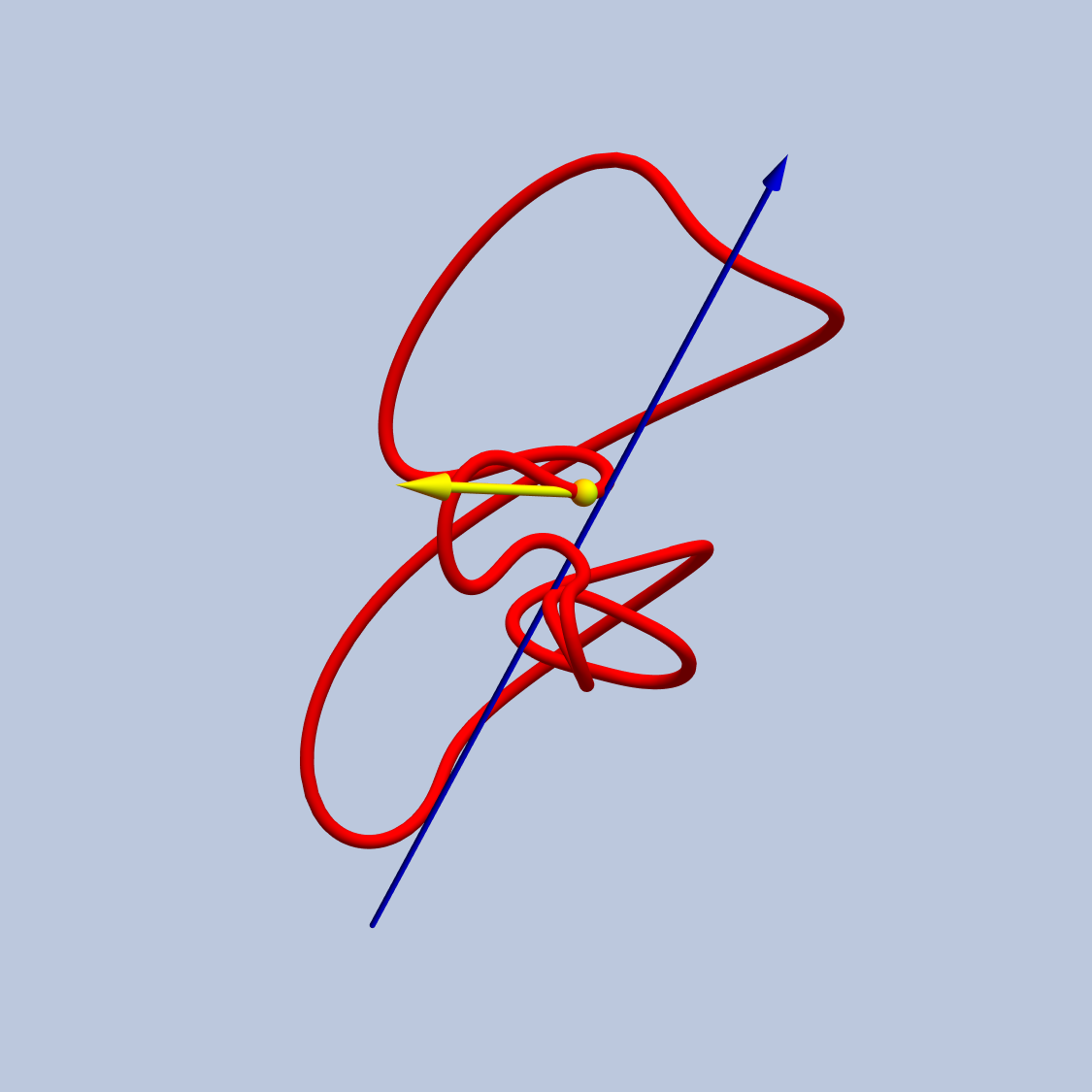}
\caption{Left: the Clifford projection $\eta_{\be'}$ of the loop with modulus $\be'=(2.65939, 3.41132)\in \Sigma_{5/6}$ The blue vector passing through the cyan point is ${\bf i}=(1,0,0)$. Right: the Heisenberg projection $(\gam_{\be'})_H$ of~$\gam_{\be'}$. The blue vector is parallel to ${\bf k}=(1,0,0)$.}\label{FIG6S4New}
\end{figure}

\item As $\be$ limits to $\be_q^{\pm}$, $\eta_{\be}$ tends to the (standard) Legendrian lifts of the corresponding multiply-covered circles.

\item If $\be\in \Sigma_q$, the Clifford projection $\eta_{\be}$ is a simple curve if and only if $m=n-1$ and $\|\be-\be_q^+\|<\varepsilon_q$, where $\varepsilon_q>0$ depends on $q=(n-1)/n$.

\item {\samepage There exist countably many $\be\in \Sigma_q$ such that the evolution $\hat{\gamma}_{\be}(-,t)$ of $\gam_{\be}$ is periodic in $t$. More precisely, let ${\mathtt P}_q\colon \Sigma_q\to \R$ be the real-analytic function
\[
{\mathtt P}_q(\be)= \frac{\pi\bigl(e_3^2-e_1^2\bigr)|\be |}{16\lambda {\rm K}(m)}
\]}%
and suppose its range is the interval ${\mathtt I}_q\subset \R$.
From Proposition~\ref{timeevolution} and Remark \ref{timeevolution2}, it follows that the evolution $\hat{\gamma}_{\be}(-,t)$ of a $\phi$-loop $\gam_{\be}\in {\mathcal G}_q$ is periodic in $t$ if and only if ${\mathtt P}(\be)=\widetilde{m}/\widetilde{n}\in {\mathbb Q}$. The time period is $2\pi\widetilde{n}n/h\lambda_{\be}$, where $h={\rm gcd}(n,\widetilde{m})$.
\end{itemize}

\begin{figure}[t]\centering
\includegraphics[height=5.4cm]{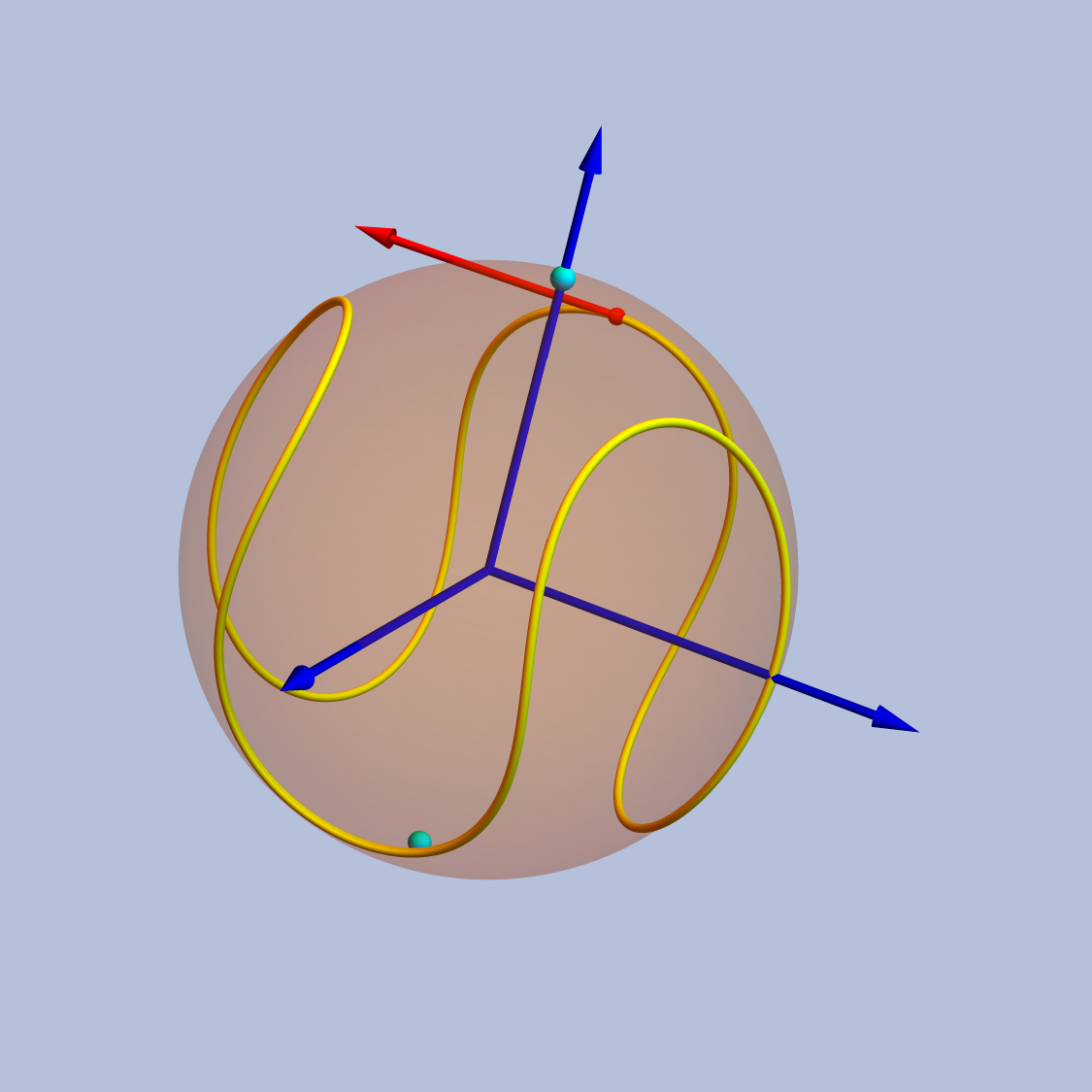}\qquad
\includegraphics[height=5.4cm]{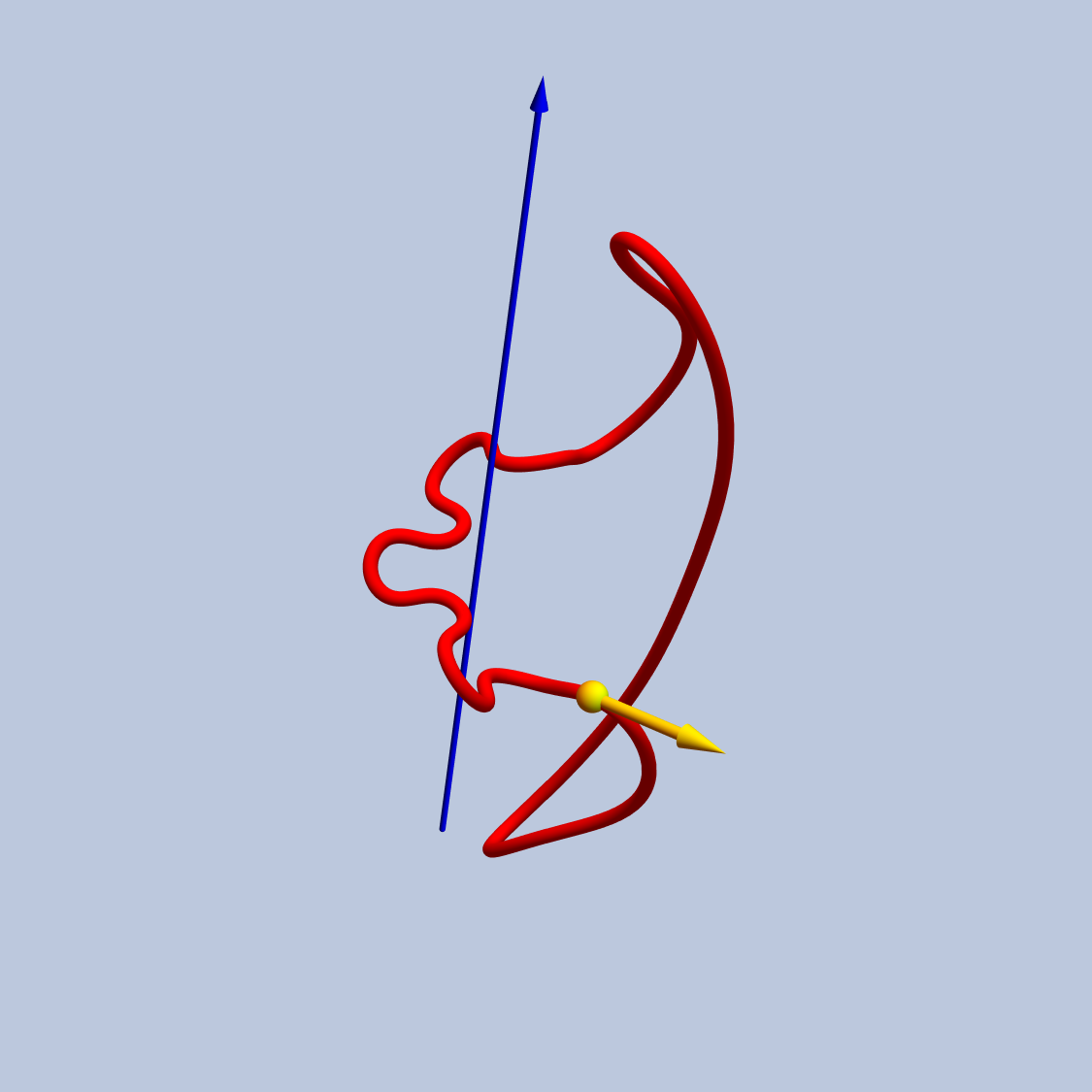}
\caption{Left: the Clifford projection $\eta_{\widetilde{\be}}$ of the loop with modulus $\widetilde{\be}=(3.99723, 5.1619)\in \Sigma_{5/6}$. The blue vector passing through the cyan point is ${\bf i}=(1,0,0)$. Right: the Heisenberg projection $(\gam_{\widetilde{\be}})_H$ of $\gam_{\widetilde{\be}}$. The blue vector is parallel to ${\bf k}=(1,0,0)$.}\label{FIG7S4New}
\end{figure}

\subsection*{Examples} We consider the isomonodromic family with $q=5/6$.
Given $\be,\be'\in \Sigma_q$, let $ \Sigma_q(\be,\be')$ be the oriented, open arc of $\Sigma_q$ joining $\be$ and $\be'$. The limit points of $\Sigma_{5/6}$ are $\be^-_{5/6}=(0,12/5)$ and $\be^+_{5/6}=(0,12)$. The exceptional point is $\be_{5/6}^*=(2.39412,3.2044)$.
\begin{itemize}\itemsep=0pt
\item Let $\be\in \Sigma_{5/6}\bigl(\be^-_{5/6},\be^*_{5/6}\bigr)$. Then the image of $\eta_{\be}$ is contained in $ \rS^2\setminus \{{\rm N}^{\pm}\}$ and has 12 ordinary double points (see Figure \ref{FIG3S4New}). Since $\eta_{\be}(s)$ for $s\in [0,3\omega_{\be}]$ turns clockwise $5$ times around the $x$-axis, then $[\eta_{\be}]\in \pi_1\bigl(\rS^2\setminus {\rm N}^{\pm}\bigr)$ is $-5$. Meanwhile, the image of the Heisenberg projection $(\gam_{\be})_H$ is a Legendrian unknot contained in $\R^3\setminus Oz$ and $(\gam_{\be})_H(s)$, $s\in [0,6\omega_{\be}]$, turns clockwise $5$ times around the upward-oriented vertical axis. Thus, $[(\gam_{\be})_H]\in \pi_1\bigl(\R^3\setminus Oz\bigr)$ is $-5$.

\item For the exceptional modulus $\be^*_{5/6}$, $\eta_{\be^*_{5/6}}$ passes three times through the poles of $\rS^2$ and has six additional ordinary double points, while
$\bigl(\gam_{\be^*_{5/6}}\bigr)_H$ intersects the $Oz$-axis in six points (see Figure \ref{FIG4S4New}).

\item Let $\be\in \Sigma_{5/6}\bigl(\be^*_{5/6},\be^+_{5/6}\bigr)$. The image of
$\eta_{\be}$ is contained in $\rS^2\setminus \{{\rm N}^{\pm}\}$ and $\eta_{\be}(s)$, $s\in [0,3\omega_{\be}]$, turns counterclockwise once around the $x$-axis, so that $[\eta_{\be}]\in \pi_1\bigl(\rS^2\setminus \{{\rm N}^{\pm}\}\bigr)$ is $1$.
Meanwhile, the image of the Heisenberg projection $(\gam_{\be})_H$ is a Legendrian unknot
contained in $\R^3\setminus Oz$ and $(\gam_{\be})_H(s)$, $s\in [0,6\omega_{\be}]$, turns counterclockwise once around the upward-oriented vertical axis. Hence, $[(\gam_{\be})_H]\in \pi_1\bigl(\R^3\setminus Oz\bigr)$ is $1$.

\item
Let $\be'=(2.65939, 3.41132)\in \Sigma_{5/6}$. $\eta_{\be'}$ has six tangential double points (see Figure \ref{FIG6S4New}).
 If~$\be\in \Sigma_{5/6}\bigl(\be',\be^+_{5/6}\bigr)$, $\eta_\be$ is a simple curve (see Figure \ref{FIG7S4New}).

\item The function ${\mathtt P}_{5/6}$ has an absolute minimum ${\mathtt p}_- \!\approx\! 0.408156$
attained at $\be^{\dag}\!\approx\! (2.7904, 3.5253)$.
The function is decreasing on $\Sigma_{5/6}\bigl(\be^-_{5/6},\be^{\dag}\bigr)$, increasing on
$\Sigma_{5/6}\bigl(\be^{\dag},\be^+_{5/6}\bigr)$ and tends to ${\mathtt p}_+\approx 6.75$ as $\be\rightarrow \be^+_{5/6}$. Thus, for every $p=\widetilde{m}/\widetilde{n}\in ({\mathtt p}_- ,{\mathtt p}_+)\cap {\mathbb Q}$
there exist a~time-periodic ${\gam}_{\check{\be}_p}\in {\mathcal G}_q$ whose evolution has period $12\pi\widetilde{n}/{\rm gcd}(6,\widetilde{m})\lambda_{\be_p}$.
For instance, let $\check{{\bf e }}=(3.245612, 10.568031)\in \Sigma_{5/6}$, then ${\mathtt P}_{5/6}(\check{{\bf e }})=5$. Hence, $\gam_{\check{{\bf e }}}$ is time-periodic and
the time-period of $\widehat{\gam}_{\check{{\bf e }}}$
is $6\tau$, $\tau=0.229849$.
Figure \ref{FIG11S4New} depicts the evolving curves $\widehat{\gamma}_{\check{{\bf e }}}(-,t)$, $t=0, 0.05,0.1,0.12$, and the orbit $t\to \widehat{\gam}_{\check{{\bf e }}}(0,t)$.
\end{itemize}

 \begin{figure}[t]\centering
\includegraphics[height=5.4cm]{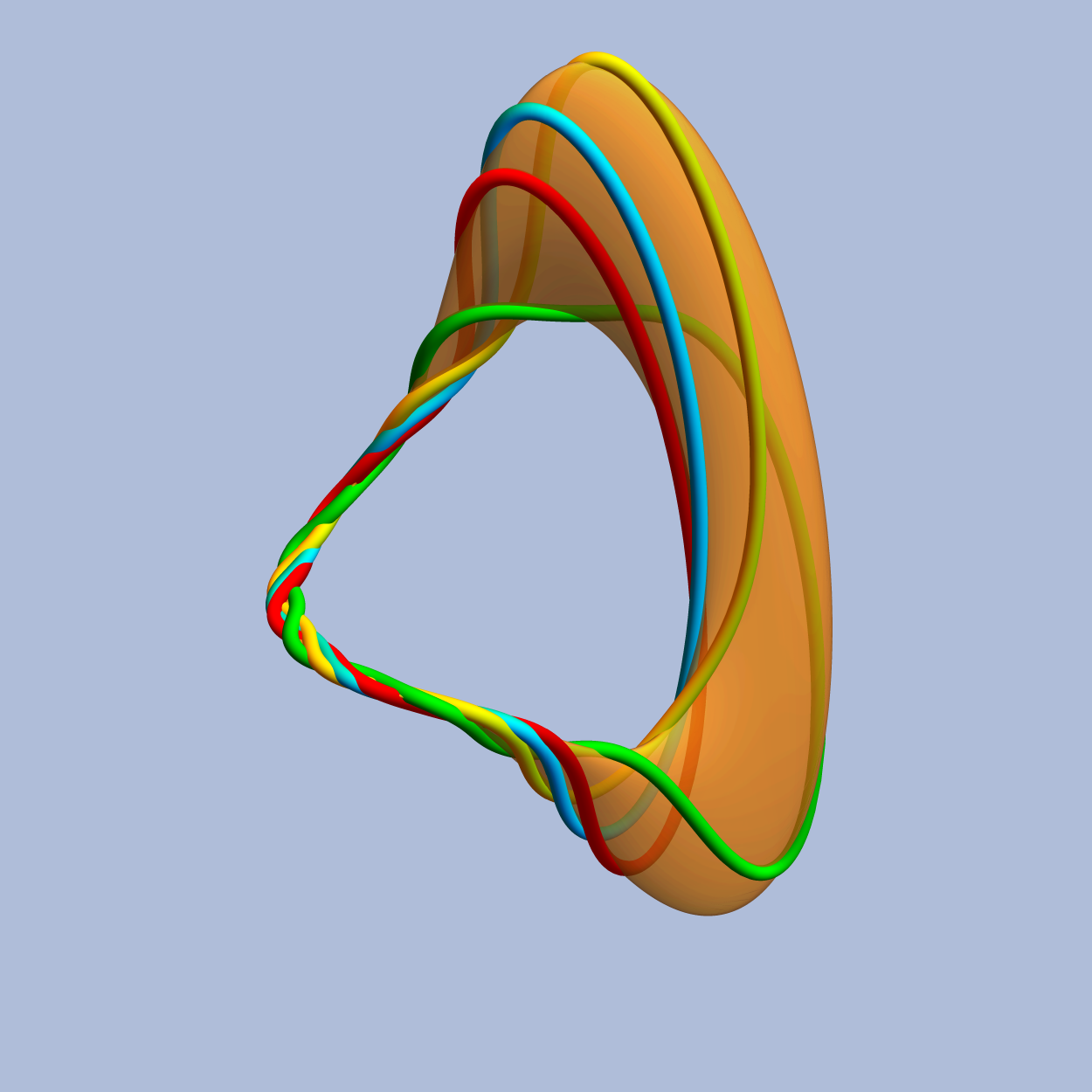}\qquad
\includegraphics[height=5.4cm]{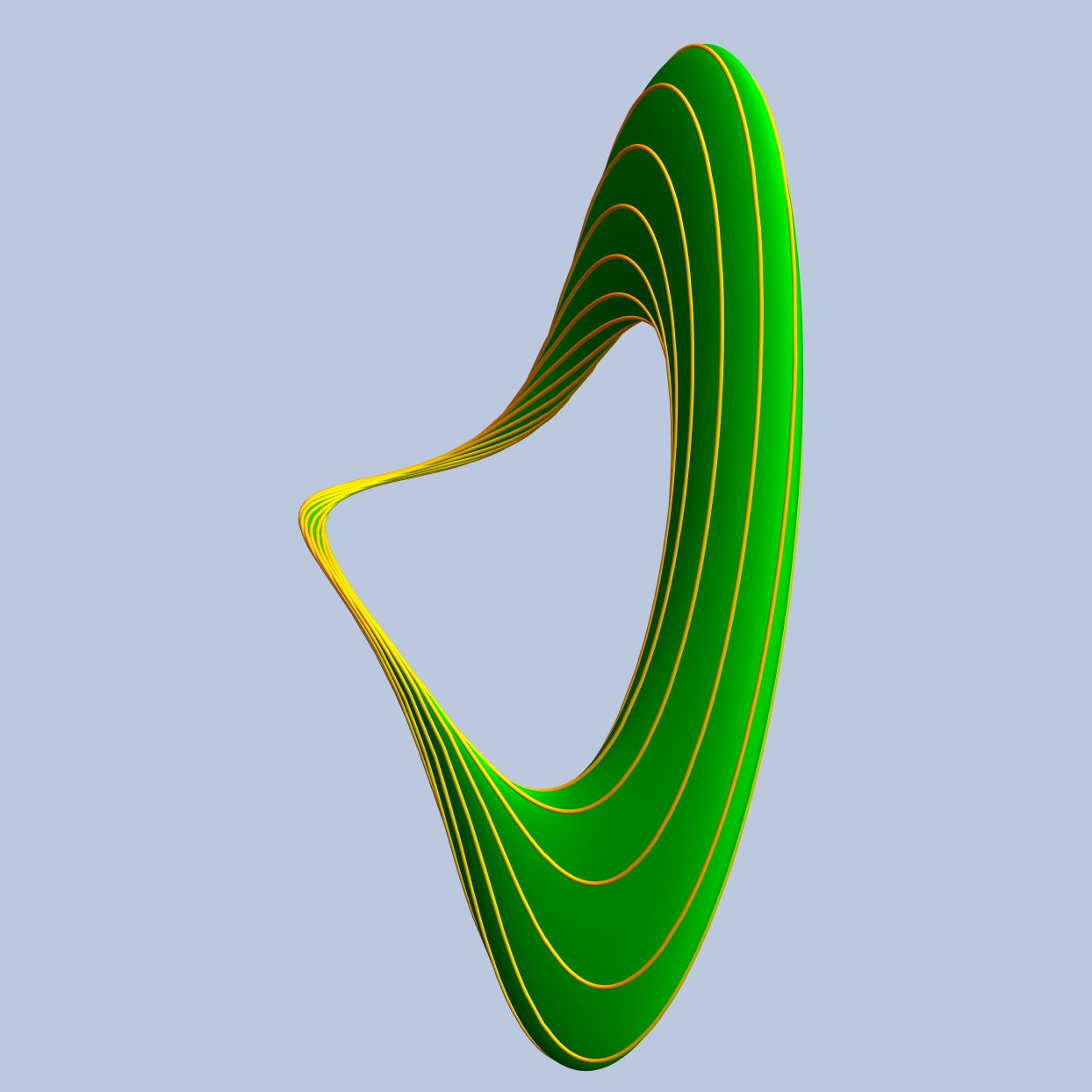}
\caption{$\widehat{\gamma}_{\check{{\bf e }}}(-,t)$, $t=0, 0.05,0.1,0.12$ and the orbit $t\to \widehat{\gam}_{\check{{\bf e }}}(0,t)$.}\label{FIG11S4New}
\end{figure}

\section{Discussion}\label{sec5}
To summarize, we have shown that, in the context of pseudo-Hermitian geometry on $S^3$,
there are flows for Legendrian curves that induce curvature evolution by any integrable PDE in the mKdV hierarchy. Moreover, we constructed a natural symplectic structure
on the space of periodic Legendrian curves relative to which each of these flows is Hamiltonian. For the flow $\VZ_1$ which induces evolution by the mKdV equation itself, we have carried out a detailed analysis of curves that are stationary (i.e.,~whose flows are congruent to the initial curve), identifying closure conditions and obtaining a complete description of periodic stationary curves in a~significant special case. These results naturally suggest further questions and directions for research.\looseness=1

First, it is worth highlighting that the closure conditions for stationary curves could
only be obtained because these curves are integrable by quadratures (see Section~\ref{quadsec}). In fact, $\VZ_1$-stationary curves arise as projections to $S^3$ of trajectories of a completely integrable contact-Hamiltonian system on ${\rm U}(2)\times \R^3$.
It is natural to ask if this holds for higher flows in the hierarchy. That is, for $n>1$
are $\VZ_n$-stationary curves also the projections of the trajectories of some completely integrable finite-dimensional contact-Hamiltonian system?

Next, as a completely integrable PDE the mKdV equation has a rich structure, including infinitely many conservation laws (as mentioned above), but also a B\"acklund transformation which generates new solutions from old ones \cite{KW,Wa}. Since this transformation can be realized as a gauge transformation or `dressing' at the level of the Lax pair, and these Lax equations are equivalent to the
AKNS-type system satisfied by our ${\rm U}(2)$-valued Frenet frame, it is natural to ask if
there is a geometric transformation for Legendrian curves that corresponds to the
B\"acklund transformation for curvature. If available, this transformation could be
used to generate new and interesting solutions to our flows, starting with some of the stationary curves we have obtained above.

Last, the pseudo-Hermitian 3-sphere has, as mentioned in the introduction, a non-compact dual $A^3$, the 3-dimensional
anti-de~Sitter space equipped with its pseudo-Hermitian structure of constant Webster curvature $-1$. In this case, we expect that the defocusing mKdV equation and its associated hierarchy can be realized by flows of Legendrian curves in $A^3$. It is also possible that there are integrable flows for null curves in this space, which would likely be related to the KdV hierarchy. If this is the case, it is natural to ask whether the Miura transformation \cite{Mi} between KdV and defocusing mKdV equations has some geometric realization that mediates between flows of Legendrian and null curves in $A^3$.

\subsection*{Acknowledgements}
We thank Cornelia Vizman for very helpful discussions of the hat calculus and the reviewers for their useful comments and suggestions. The authors gratefully acknowledge the support and warm hospitality of the Department of Mathematics at the Politecnico di Torino.
A.~Calini would also like to thank the Isaac Newton Institute for Mathematical Sciences for support and hospitality during the programme {\em Dispersive hydrodynamics: mathematics, simulation and experiments, with applications in nonlinear waves} when part of this work was undertaken.
This work was partially supported by PRIN 2017 {\em Real and Complex Manifolds: Topology, Geometry and holomorphic dynamics} (protocollo 2017JZ2SW5-004);
by the Gruppo Nazionale per le Strutture Algebriche, Geometriche e le loro Applicazioni within the Istituto Nazionale di Alta Matematica ``Francesco Severi'' (INdAM); and by EPSRC grant number EP/R014604/1.

\pdfbookmark[1]{References}{ref}
\LastPageEnding

\end{document}